\documentclass[11pt]{article}
\usepackage[utf8]{inputenc}
\usepackage[T1]{fontenc}
\usepackage{amsmath, amssymb, amsthm}
\usepackage{hyperref}
\usepackage{tikz}
\usetikzlibrary{arrows.meta,positioning}
\usepackage{amscd}
\usetikzlibrary{arrows.meta,positioning,shapes.geometric,calc}
\usepackage{xurl} 
\usepackage{etoolbox}
\apptocmd{\thebibliography}{\sloppy}{}{} 

\definecolor{vertcol}{RGB}{30,120,180}
\definecolor{horizcol}{RGB}{180,70,30}

\usepackage{url}

\providecommand{\keywords}[1]{\textbf{Keywords:} #1}
\providecommand{\subjclass}[1]{\textbf{2020 Mathematics Subject Classification:} #1}

\newtheorem{theorem}{Theorem}[section]
\newtheorem{lemma}[theorem]{Lemma}
\newtheorem{proposition}[theorem]{Proposition}
\newtheorem{corollary}[theorem]{Corollary}
\newtheorem{conjecture}[theorem]{Conjecture}
\newtheorem{remark}[theorem]{Remark}
\theoremstyle{definition}
\newtheorem{definition}[theorem]{Definition}
\newtheorem{example}[theorem]{Example}

\newcommand{\Jac}{\operatorname{Jac}}
\newcommand{\Ham}{\operatorname{Ham}}

\begin{document}
	
	\title{Abel--Jacobi Map and Symplectic Topology}
	\author{St\'ephane Tchuiaga\\Department of Mathematics, University of Buea\\
		\texttt{tchuiaga.kameni@ubuea.cm}}
	\maketitle
	
	\begin{abstract}

		We develop a harmonic-coordinate approach to the identity component $G_\omega(\Sigma_g)$ of the symplectomorphism group of a closed oriented surface of genus $g\ge2$. Using an intrinsic decomposition of the Abel--Jacobi displacement into a global flux part and a zero-average harmonic fluctuation, we introduce the harmonic flux norm and prove its non-degeneracy on the full identity component without Floer theory. We also show the norm is continuous in the $C^0$-topology, deduce that $\mathrm{Ham}(\Sigma_g,\omega)$ is $C^0$-closed inside $G_\omega(\Sigma_g)$, and produce a locally injective harmonic-coordinate chart near the identity. Along the way we derive first-order expansions for the norm, quantitative fixed-point obstructions, and propose a finite-dimensional persistence invariant (the harmonic barcode) associated to the harmonic displacement filtration.
	\end{abstract}
	
	

	\section{Introduction}
	The study of symplectomorphism groups lies at the heart of modern symplectic topology. For a closed symplectic manifold $(M,\omega)$, the identity component $G_\omega(M)$ of the group of symplectomorphisms carries a rich geometric and topological structure. Understanding the interplay between its algebraic, dynamical, and metric properties has been a central theme since the foundational work of Banyaga \cite{Banyaga1997}. Among the most important invariants associated with $G_\omega(M)$ is the \emph{flux homomorphism}, which measures the cohomological displacement of points under a symplectic isotopy. For closed surfaces of genus $g\ge 2$, a classical theorem asserts that the flux group vanishes \cite{Banyaga1997}. Consequently, the flux homomorphism becomes a surjection
	\[
	\operatorname{Flux}: G_\omega(\Sigma_g) \longrightarrow H^1(\Sigma_g;\mathbb{R}),
	\]
	whose kernel is precisely the Hamiltonian subgroup $\operatorname{Ham}(\Sigma_g,\omega)$. While the flux provides crucial cohomological information, it does not measure the \emph{pointwise} displacement of individual points on the surface. This gap motivates the search for a finer invariant that simultaneously captures global topological and local geometric data.
	
	\medskip\noindent
	In the erratum \cite{Tchuaga2021Erratum}, we introduced a novel invariant---the \emph{harmonic flux norm}---which refines the classical flux homomorphism by incorporating Hodge-theoretic data. For a closed oriented surface $(\Sigma,\omega)$, this norm is defined by
	\[
	\|\phi\|_{\mathrm{harm}} := \sup_{\substack{\alpha\in\mathcal H^1(\Sigma)\\ \|\alpha\|_{L^2}=1}} \sup_{x\in \Sigma} |\Delta(\phi,\alpha)_x|,
	\]
	where $\mathcal H^1(\Sigma)$ denotes the finite-dimensional space of harmonic $1$-forms (with respect to a chosen Riemannian metric compatible with $\omega$), and
	\[
	\Delta(\phi,\alpha)_x = \langle [\alpha], \widetilde{S}_\omega(\Phi) \rangle - \operatorname{Vol}(\Sigma) \int_{\mathcal O_x^\Phi} \alpha
	\]
	for any symplectic isotopy $\Phi$ from $\mathrm{id}$ to $\phi$. Here, $\mathcal O_x^\Phi$ denotes the orbit of the point $x$ under the isotopy, and the integral measures the signed area swept out by this orbit weighted by the harmonic form $\alpha$. The first term records the global flux contribution, while the second encodes local displacement.
	
	While it was shown in \cite{Tchuaga2021Erratum} that this quantity is always \emph{finite}, the question of whether it is \emph{non-degenerate}---that is, whether $\|\phi\|_{\mathrm{harm}}=0$ implies $\phi=\mathrm{id}$---remained open. The main goal of this paper is to settle this conjecture affirmatively and to do so using a purely geometric construction that bypasses infinite-dimensional Floer-theoretic machinery.
	
	\medskip\noindent
	Our approach rests on the classical Abel--Jacobi map. For a closed oriented surface $\Sigma_g$ of genus $g\ge2$, fix a basepoint $x_0\in\Sigma_g$ and an $L^2$-orthonormal basis $\{h_1,\ldots,h_{2g}\}$ of harmonic $1$-forms. The Abel--Jacobi map
	\[
	\mathcal A:\Sigma_g \longrightarrow \operatorname{Jac}(\Sigma_g):=\mathbb R^{2g}/\Gamma,
	\qquad
	\mathcal A(x)=\left(\int_{x_0}^x h_1,\dots,\int_{x_0}^x h_{2g}\right)\bmod\Gamma,
	\]
	is a holomorphic embedding, where $\Gamma$ is the period lattice determined by the basis. The key insight is that the displacement of the Abel--Jacobi image under a symplectomorphism admits a canonical decomposition. For any lift $\widetilde\phi$ of $\phi$ to the universal cover and any lift $\widetilde{\mathcal A}$ of the Abel--Jacobi map, we have
	\[
	\widetilde{\mathcal A}(\widetilde\phi(\tilde x))-\widetilde{\mathcal A}(\tilde x)
	=
	\widetilde{\mathbf c}(\phi)+\widetilde{\mathcal F}(\phi)(\tilde x),
	\]
	where $\widetilde{\mathbf c}(\phi)\in\mathbb R^{2g}\cong H^1(\Sigma_g;\mathbb R)$ is a constant vector representing the flux class, and $\widetilde{\mathcal F}(\phi)$ has zero average over a fundamental domain. The term $\widetilde{\mathbf c}(\phi)$ encodes the global cohomological displacement, while $\widetilde{\mathcal F}(\phi)$ captures the pointwise harmonic fluctuation. This decomposition separates the cohomological and local geometric data of the symplectomorphism, providing a natural framework for studying its dynamics.
	
	\medskip\noindent
	Using this framework, we establish the following three principal results. First, we resolve the non-degeneracy conjecture for the harmonic flux norm:
	
	\begin{theorem}[Non-degeneracy of the Harmonic Flux Norm]
		\label{thm:intro_nondeg}
		Let $(\Sigma_g,\omega)$ be a closed oriented surface of genus $g\ge2$. Then the harmonic flux norm is non-degenerate on the full identity component of the symplectomorphism group:
		\[
		\|\phi\|_{\mathrm{harm}}=0 \quad\Longleftrightarrow\quad \phi=\mathrm{id}.
		\]
	\end{theorem}
	
	The proof exploits the rigidity of the Abel--Jacobi image: for $g\ge2$, the image $\mathcal A(\Sigma_g)\subset \operatorname{Jac}(\Sigma_g)$ admits no non-trivial translation symmetries. If the harmonic fluctuation vanishes, the Abel--Jacobi displacement is globally constant. The surface is then translated within its Jacobian; since the image has no translation symmetry, this constant must be zero. Injectivity of $\mathcal A$ then forces $\phi=\mathrm{id}$.
	
	Second, we establish that the harmonic flux norm is well-behaved with respect to the $C^0$-topology on $G_\omega(\Sigma_g)$:
	
	\begin{theorem}[$C^0$-Continuity and Closedness]
		\label{thm:intro_continuity}
		The harmonic flux norm $\|\cdot\|_{\mathrm{harm}}: G_\omega(\Sigma_g) \longrightarrow [0,\infty)$ is continuous with respect to the $C^0$-topology. Consequently, $\operatorname{Ham}(\Sigma_g,\omega)$ is a $C^0$-closed subgroup of $G_\omega(\Sigma_g)$.
	\end{theorem}
	
	This gives a new proof of a classical theorem in $C^0$-symplectic topology: the Hamiltonian subgroup is closed in the uniform topology. Unlike the standard proof, which relies on the theory of $C^0$-flux, our argument follows directly from the continuity of the harmonic coordinate map.
	
	Third, we show that the harmonic fluctuation component defines a local coordinate chart near the identity:
	
	\begin{theorem}[Local Coordinate Embedding]
		\label{thm:intro_injectivity}
		There exists a $C^0$-neighborhood $\mathcal{U}$ of the identity in $G_\omega(\Sigma_g)$ such that the harmonic fluctuation map
		\[
		\mathcal{F}: \mathcal{U} \longrightarrow C^0_0(\Sigma_g; \mathbb{R}^{2g}), \qquad \phi \mapsto \mathcal{F}(\phi)
		\]
		is injective and continuous.
	\end{theorem}
	
	This result provides a canonical embedding of a neighborhood of the identity into a Banach space, mirroring the role of the flux homomorphism in providing local coordinates near the identity in the symplectomorphism group.
	
	\medskip\noindent
	The harmonic decomposition yields additional geometric insights. The fixed point equation for a symplectomorphism can be reformulated as a global level-set problem for the harmonic primitive: a point $x\in\Sigma_g$ is fixed by $\phi$ if and only if $\mathcal F(\phi)(x)+\mathbf c(\phi)=0$ in $\operatorname{Jac}(\Sigma_g)$. This leads to a quantitative obstruction: if the harmonic oscillation is smaller than the distance from the flux translation to the period lattice, then $\phi$ has no fixed points. Thus the harmonic norm measures the minimal fluctuation required to cancel a given cohomological displacement.
	
	\medskip\noindent
	A distinguishing feature of our approach is its elementary nature. Unlike many results in $C^0$-symplectic topology, which rely on Floer homology, spectral invariants, or other infinite-dimensional techniques, our proofs exploit only the finite-dimensional geometry of the Abel--Jacobi embedding and the classical Torelli theorem. This not only makes the arguments more accessible but also suggests potential generalizations to higher-dimensional symplectic manifolds via the Albanese map. We discuss such extensions, as well as a proposed construction of harmonic barcodes using persistent homology, in the final section.
	
	\medskip\noindent
	The paper is organized as follows. Section 2 reviews the necessary preliminaries on symplectic flux theory, harmonic forms, and the Abel--Jacobi map, and proves the fundamental decomposition lemma (Lemma~\ref{lem:AJ_decomposition}). Section 3 defines the harmonic flux norm and establishes its key properties: the first-order expansion, non-degeneracy on both the Hamiltonian subgroup and the full identity component, $C^0$-continuity, and local injectivity. In Subsection~3.1, we apply the framework to the fixed point problem, deriving a geometric reformulation and a quantitative obstruction. Section 4 discusses perspectives and open problems, including extensions to topological area-preserving homeomorphisms, generalizations via the Albanese map, and the proposed harmonic barcodes.

	\section{Preliminaries}
	
	In this section we recall the fundamental notions of symplectic geometry, flux
	theory, harmonic forms, and the Abel--Jacobi map that are used throughout the
	paper \cite{LeRoux2010, Petersen2006, McDuffSalamon, Polterovich,
		HoferZehnder}. For a comprehensive treatment, we refer the reader to the
	monographs \cite{Tchuaga2018} and \cite{Banyaga1997}.
	
	\subsection{Symplectic manifolds and symplectomorphism groups}
	
	A \emph{symplectic manifold} is a pair $(M,\omega)$, where $M$ is a smooth
	manifold and $\omega$ is a closed, non-degenerate $2$-form. A diffeomorphism
	$\phi:M\to M$ satisfying $\phi^*\omega=\omega$ is called a
	\emph{symplectomorphism}. The group of all symplectomorphisms is denoted
	$\operatorname{Symp}(M,\omega)$. We equip $\operatorname{Symp}(M,\omega)$ with
	the $C^\infty$-compact open topology \cite{MH}, and its identity component is
	denoted $G_\omega(M)$. For a closed connected manifold, a smooth isotopy
	$\{\phi_t\}_{t\in[0,1]}\subset G_\omega(M)$ with $\phi_0=\mathrm{id}$ is
	generated by a time-dependent vector field $X_t$, defined by
	$\frac{d}{dt}\phi_t = X_t\circ\phi_t$. Since each $\phi_t$ preserves $\omega$,
	the vector fields satisfy $\mathcal L_{X_t}\omega=0$. A symplectic vector field
	$X$ is \emph{Hamiltonian} if $\iota_X\omega$ is exact; its flow consists of
	Hamiltonian diffeomorphisms. The group of Hamiltonian diffeomorphisms, denoted
	$\operatorname{Ham}(M,\omega)$, is a normal subgroup of $G_\omega(M)$
	\cite{Banyaga1997}.
	
	\subsubsection{Flux homomorphism}
	
	The \emph{flux homomorphism} on the universal cover
	$\widetilde{G_\omega(M)}$ is defined by
	\[
	\widetilde{\operatorname{Flux}}(\{\phi_t\}) = \int_0^1 [\iota_{X_t}\omega]\,dt
	\in H^1(M;\mathbb R),
	\]
	where $\{\phi_t\}$ is an isotopy from the identity. It descends to a
	homomorphism
	\[
	\operatorname{Flux}: G_\omega(M) \longrightarrow H^1(M;\mathbb R)/\Gamma_\omega,
	\]
	with $\Gamma_\omega$ the flux group. A classical theorem of Banyaga
	\cite{Banyaga1997} states that $\ker \operatorname{Flux} =
	\operatorname{Ham}(M,\omega)$. For a smooth closed oriented surface $\Sigma_g$
	of genus $g\geq 2$, the flux group is trivial (\cite{Banyaga1997}, §4), hence
	$\Gamma_\omega=\{0\}$. In this context, the two flux maps have the same range.
	
	\subsection{Harmonic forms and primitives}
	
	Let $(M,g)$ be a closed Riemannian manifold. The space of harmonic $1$-forms is
	\[
	\mathcal H^1(M)=\{\alpha\in\Omega^1(M): d\alpha=0,\; d^*\alpha=0\},
	\]
	where $d^*$ is the codifferential. By Hodge theory, $\mathcal H^1(M)\cong
	H^1(M;\mathbb R)$, so $\dim\mathcal H^1(M)=b_1(M)$. The space carries the
	$L^2$-inner product 
	\[
	\langle \alpha,\beta\rangle_{L^2}=\int_M \alpha\wedge\star\beta.
	\]  
	We shall denote by $\|\cdot\|_{L^2}$ the norm associated with the inner product
	$\langle \cdot,\cdot\rangle_{L^2}$.
	
	Note that for $\phi\in G_\omega(M)$ and $\alpha\in\mathcal H^1(M)$, the pullback
	$\phi^*\alpha$ is closed and cohomologous to $\alpha$. Hence, there is a unique
	zero-mean function $f_\alpha^\phi$ (with respect to the symplectic volume
	$\omega^n$) such that
	\[
	\phi^*\alpha - \alpha = d f_\alpha^\phi.
	\]
	This primitive depends only on $\phi$ and the cohomology class of $\alpha$.
	
	\subsection{The Abel-Jacobi map}
	
	For a closed oriented surface $\Sigma_g$ of genus $g\ge 2$, fix a base point
	$x_0$ and an $L^2$-orthonormal basis $\{h_1,\dots,h_{2g}\}$ of $\mathcal
	H^1(\Sigma_g)$. The Abel-Jacobi map is
	\[
	\mathcal A:\Sigma_g \longrightarrow \operatorname{Jac}(\Sigma_g):=\mathbb
	R^{2g}/\Gamma,
	\qquad
	\mathcal A(x)=\left(\int_{x_0}^x h_1,\dots,\int_{x_0}^x h_{2g}\right)\bmod\Gamma,
	\]
	where $\Gamma$ is the period lattice. For $g\ge2$, the harmonic forms have no
	common zeros, so $\mathcal A$ is a holomorphic embedding (see
	\cite{GriffithsHarris}). Moreover, a classical consequence of the Torelli
	theorem (see \cite{GriffithsHarris}, Chapter III, Section 6) states that the
	image $\mathcal A(\Sigma_g)$ admits no non-trivial translation symmetry:
	\[
	\mathcal A(\Sigma_g) + v = \mathcal A(\Sigma_g) \quad \Longrightarrow \quad v =
	0 \in \Jac(\Sigma_g).
	\]
	
	Let $\widetilde{\mathcal A}:\widetilde{\Sigma}_g\rightarrow\mathbb R^{2g}$ be
	the lifted Abel--Jacobi map associated with this basis. For $\phi\in
	G_\omega(\Sigma_g)$, choose a symplectic isotopy $\Phi=\{\phi_t\}_{t\in[0,1]}$
	from $\mathrm{id}$ to $\phi$, generated by $X_t$, and let $\widetilde\phi:
	\widetilde{\Sigma}_g\rightarrow\widetilde{\Sigma}_g$ be a lift of $\phi$.
	Define
	\[
	\widetilde{\mathbf c}(\phi)
	=
	\frac1{\operatorname{Vol}(\Sigma_g)}
	\int_D
	\left(
	\widetilde{\mathcal A}(\widetilde\phi(\tilde x))
	-
	\widetilde{\mathcal A}(\tilde x)
	\right)
	\,d\operatorname{Vol},
	\]
	where $D$ is a fundamental domain of the universal covering. Then
	$\widetilde{\mathbf c}(\phi)\in \mathbb R^{2g} \simeq H^1(\Sigma_g;\mathbb
	R)$ is a lift of the flux class.

	The following lemma shows that the displacement induced by a symplectomorphism
	splits into a constant global shift $\widetilde{\mathbf c}(\phi)$ (the "Flux
	component") and a point-dependent wiggle $\widetilde{\mathcal F}(\phi)$ (the
	"fluctuation").
	
	\begin{lemma}[Decomposition of the Abel--Jacobi displacement]
		\label{lem:AJ_decomposition}
		Let $(\Sigma_g,\omega)$ be a closed oriented surface of genus $g\ge2$, and let
		$\{h_1,\ldots,h_{2g}\}$ be an $L^2$-orthonormal basis of harmonic $1$-forms.
		Then the following hold:
		
		\begin{enumerate}
			\item
			There is a unique decomposition
			\[
			\widetilde{\mathcal A}(\widetilde\phi(\tilde x))
			-
			\widetilde{\mathcal A}(\tilde x)
			=
			\widetilde{\mathbf c}(\phi)
			+
			\widetilde{\mathcal F}(\phi)(\tilde x),
			\]
			where $\widetilde{\mathbf c}(\phi)\in\mathbb R^{2g}$ is constant and $\int_D
			\widetilde{\mathcal F}(\phi)(\tilde x)\, d\operatorname{Vol} =0.$
			
			\item
			The vector $\widetilde{\mathbf c}(\phi)$ is independent of the choice of lift
			$\widetilde\phi:\widetilde\Sigma_g\to\widetilde\Sigma_g.$
			
			\item
			The class $\mathbf c(\phi) := \widetilde{\mathbf c}(\phi)\bmod\Gamma \in
			\operatorname{Jac}(\Sigma_g)$ depends only on the endpoint $\phi$ and is
			independent of the chosen symplectic isotopy joining $\mathrm{id}$ to
			$\phi$.	
			
			\item
			Since the flux group satisfies $\Gamma_\omega=\{0\},$ then $\widetilde{\mathbf
				c}(\phi)$ is uniquely determined by $\phi$ and is the unique lift of $\mathbf
			c(\phi).$
			
			\item 	For all $\phi,\psi\in G_\omega(\Sigma_g)$,
			\[
			\widetilde{\mathbf c}(\phi\circ\psi)
			=
			\widetilde{\mathbf c}(\phi)
			+
			\widetilde{\mathbf c}(\psi).
			\]
			
			\item
			If $\phi\in\Ham(\Sigma_g,\omega),$ then $\widetilde{\mathbf c}(\phi)=0.$
			
		\end{enumerate}
	\end{lemma}
	
	\begin{proof}
		Let $\Phi=\{\phi_t\}_{t\in[0,1]}$ be a symplectic isotopy from $\mathrm{id}$ to
		$\phi$, generated by the time-dependent vector field $X_t$, so that
		$\frac{d}{dt}\phi_t=X_t\circ\phi_t.$ The corresponding flux form is
		$\beta_\Phi=\int_0^1\iota_{X_t}\omega\,dt,$ and the flux class is
		$\operatorname{Flux}(\Phi)=[\beta_\Phi]\in H^1(\Sigma_g;\mathbb R).$ Let
		$\widetilde\Phi = \{\widetilde\phi_t\}_{t\in[0,1]}$ be the lift of the isotopy
		to the universal cover, with $\widetilde\phi_0=\mathrm{id}$. For each harmonic
		form $h_i$, by the definition of the lifted Abel--Jacobi map we have
		\[
		\frac{d}{dt}
		\widetilde{\mathcal A}_i
		(\widetilde\phi_t(\tilde x))
		=
		\pi^*h_i
		\left(
		\frac{d}{dt}\widetilde\phi_t(\tilde x)
		\right).
		\]
		
		Since $\widetilde\Phi$ is the lift of $\Phi$, the lifted vector field is the
		pullback of $X_t$. Hence, writing $x=\pi(\tilde x)$,
		\[
		\frac{d}{dt} \widetilde{\mathcal A}_i (\widetilde\phi_t(\tilde x)) =
		h_i(X_t)(\phi_t(x)).
		\]
		Integrating in time gives
		\[
		\widetilde{\mathcal A}_i(\widetilde\phi(\tilde x))
		-
		\widetilde{\mathcal A}_i(\tilde x)
		=
		\int_0^1 h_i(X_t)(\phi_t(x))\,dt.
		\tag{A}
		\]
		
		Let $D\subset\widetilde{\Sigma}_g$ be a fundamental domain. Since $\pi:D
		\rightarrow\Sigma_g$ identifies $D$ with $\Sigma_g$ up to its boundary,
		integration over $D$ gives
		\[
		\int_D
		\left(
		\widetilde{\mathcal A}_i(\widetilde\phi(\tilde x))
		-
		\widetilde{\mathcal A}_i(\tilde x)
		\right)d\operatorname{Vol}
		=
		\int_0^1
		\int_{\Sigma_g}
		h_i(X_t)(\phi_t(x))
		\,\omega\,dt.
		\tag{B}
		\]
		
		Because $\phi_t$ is symplectic, $\phi_t^*\omega=\omega.$ Therefore, using the
		change of variables $y=\phi_t(x)$,
		\[
		\int_{\Sigma_g}
		h_i(X_t)(\phi_t(x))\,\omega
		=
		\int_{\Sigma_g}
		h_i(X_t)(y)\,\omega.
		\tag{T}
		\]
		Now we use the pointwise identity on a surface. Let $\eta=\iota_X\omega,$ for
		any one-form $h$, one has $\eta\wedge h=h(X)\omega.$ Indeed, for vectors $u,v$,
		\[
		(\eta\wedge h)(u,v)
		=
		\eta(u)h(v)-\eta(v)h(u)
		=
		\omega(X,u)h(v)-\omega(X,v)h(u),
		\]
		and since $\omega$ is the volume form on the surface this equals
		$h(X)\omega(u,v)$. Hence, $h_i(X_t)\omega = \iota_{X_t}\omega\wedge h_i.$
		Consequently,
		\[
		\int_{\Sigma_g}h_i(X_t)\omega
		=
		\int_{\Sigma_g}
		\iota_{X_t}\omega\wedge h_i.
		\tag{S}
		\]
		Combining (B), (T), and (S), we obtain
		\[
		\begin{aligned}
			\frac1{\operatorname{Vol}(\Sigma_g)}
			\int_D
			\left(
			\widetilde{\mathcal A}_i(\widetilde\phi(\tilde x))
			-
			\widetilde{\mathcal A}_i(\tilde x)
			\right)
			d\operatorname{Vol}
			=
			\frac1{\operatorname{Vol}(\Sigma_g)}
			\int_0^1
			\left(
			\int_{\Sigma_g}
			\iota_{X_t}\omega\wedge h_i
			\right)dt.
		\end{aligned}
		\tag{K}
		\]
		
		The right-hand side is exactly the harmonic coordinate of the cohomology class
		$\left[\int_0^1\iota_{X_t}\omega\,dt\right] = \operatorname{Flux}(\Phi).$ That
		is (see \cite{Tchuaga2018}), 
		\[
		c_i(\phi) = \frac{1}{\operatorname{Vol}(\Sigma_g)} \langle
		\operatorname{Flux}(\Phi), [h_i]\rangle_{L^2}.
		\]
		
		Therefore the vector
		\[
		\widetilde{\mathbf c}(\phi)
		:=
		\frac1{\operatorname{Vol}(\Sigma_g)}
		\int_D
		\left(
		\widetilde{\mathcal A}(\widetilde\phi(\tilde x))
		-
		\widetilde{\mathcal A}(\tilde x)
		\right)
		d\operatorname{Vol}
		\]
		is precisely the harmonic-coordinate representative of the flux class
		$\operatorname{Flux}(\Phi)$ in $H^1(\Sigma_g;\mathbb R)$. Now define the
		fluctuation term by 
		\[
		\widetilde{\mathcal F}(\phi)(\tilde x)
		:=
		\widetilde{\mathcal A}(\widetilde\phi(\tilde x))
		-
		\widetilde{\mathcal A}(\tilde x)
		-
		\widetilde{\mathbf c}(\phi).
		\]
		
		By construction, $\int_D \widetilde{\mathcal F}(\phi)(\tilde x)
		\,d\operatorname{Vol} =0.$ Hence,
		\[
		\widetilde{\mathcal A}(\widetilde\phi(\tilde x))
		-
		\widetilde{\mathcal A}(\tilde x)
		=
		\widetilde{\mathbf c}(\phi)
		+
		\widetilde{\mathcal F}(\phi)(\tilde x)
		\]
		is the desired decomposition. The uniqueness follows immediately. Suppose that
		$\widetilde{\mathcal A}(\widetilde\phi(\tilde x)) - \widetilde{\mathcal
			A}(\tilde x) = c+F(\tilde x),$ with $\int_D F\,d\operatorname{Vol}=0.$
		Integrating over $D$ gives 
		\[
		c
		=
		\frac1{\operatorname{Vol}(\Sigma_g)}
		\int_D
		\left(
		\widetilde{\mathcal A}(\widetilde\phi(\tilde x))
		-
		\widetilde{\mathcal A}(\tilde x)
		\right)
		d\operatorname{Vol},
		\]
		so necessarily $c=\widetilde{\mathbf c}(\phi).$ If $\Phi'$ is another symplectic
		isotopy from $\mathrm{id}$ to $\phi$, then $\operatorname{Flux}(\Phi') -
		\operatorname{Flux}(\Phi) \in \Gamma_\omega = \{0\},$ because the concatenation
		$\Phi'*\Phi^{-1}$ is a loop in $G_\omega(\Sigma_g)$. Hence the class of
		$\widetilde{\mathbf c}(\phi)$ in the quotient is independent of the chosen
		isotopy. Thus, the Abel--Jacobi displacement splits into a constant translation
		part, and a zero-mean harmonic fluctuation. Item (3) follows immediately from
		formula (K) because the RHS is exactly $\frac1{\operatorname{Vol}(\Sigma_g)}
		\langle Flux(\Phi), [h_i]\rangle_{L^2}$ \cite{Tchuaga2018}, and $
		Flux(\Phi) = 0$ provided $\Phi$ is Hamiltonian. This does not depend on the
		choice of any isotopy $\Phi$ from the identity to $\phi$ as $\Gamma_\omega =
		\{0\}$.
		
		For (5): Let $\Phi=\{\phi_t\}_{t\in[0,1]}$ and $\Psi=\{\psi_t\}_{t\in[0,1]}$ be
		symplectic isotopies joining the identity to $\phi$ and $\psi$, generated by the
		time-dependent vector fields $X_t$ and $Y_t$, respectively. Define the
		composition isotopy $\Xi=\{\xi_t\}_{t\in[0,1]}$, where $\xi_t=\phi_t\circ\psi_t.$
		Then, $\Xi$ joins $\mathrm{id}$ to $\phi\circ\psi$. Differentiating
		$\xi_t=\phi_t\circ\psi_t$ gives
		\[
		\frac{d}{dt}\xi_t
		=
		\frac{d}{dt}\phi_t\circ\psi_t
		+
		D\phi_t(\psi_t)
		\frac{d}{dt}\psi_t.
		\]
		Since
		\[
		\frac{d}{dt}\phi_t
		=
		X_t\circ\phi_t,
		\qquad
		\frac{d}{dt}\psi_t
		=
		Y_t\circ\psi_t,
		\]
		we obtain
		\[
		\frac{d}{dt}\xi_t
		=
		X_t\circ\xi_t
		+
		(\phi_t)_*Y_t\circ\xi_t.
		\]
		Hence, the generating vector field of $\Xi$ is $Z_t = X_t+(\phi_t)_*Y_t.$
		Therefore, $\operatorname{Flux}(\Xi) = \operatorname{Flux}(\Phi) +
		\operatorname{Flux}(\Psi)$ \cite{Banyaga1997}. By formula (K), the vector
		$\widetilde{\mathbf c}(\eta)$ is exactly the harmonic-coordinate representative
		of the flux class of any symplectomorphism $\eta\in G_\omega(\Sigma_g)$. Since
		the flux group $\Gamma_\omega=\{0\}$ for closed surfaces of genus $g\ge2$, this
		representative is unique. Hence,
		\[
		\widetilde{\mathbf c}(\phi\circ\psi)
		=
		\widetilde{\mathbf c}(\phi)
		+
		\widetilde{\mathbf c}(\psi),
		\]
		which proves the homomorphism property.
	\end{proof}
	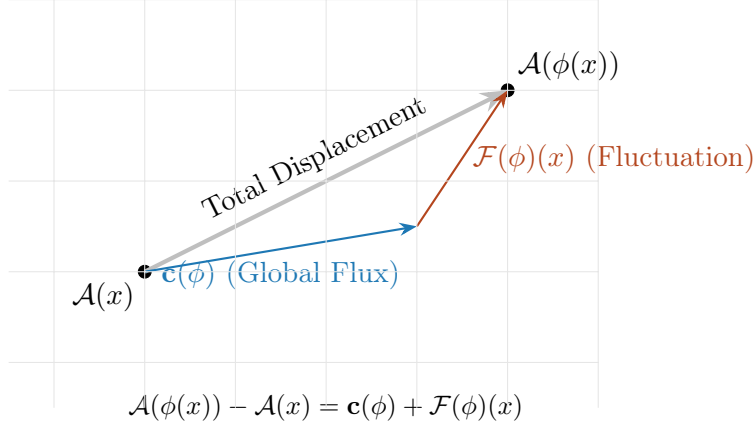
\begin{figure}[htbp]
		\centering
		\begin{tikzpicture}[scale=1.2, >=Stealth]
			\coordinate (O) at (0,0);
			\coordinate (Ax) at (1,1);
			\coordinate (Aphi) at (5,3);
			\coordinate (Mid) at (4,1.5);
			
			\filldraw (Ax) circle (2pt) node[below left] {$\mathcal{A}(x)$};
			\filldraw (Aphi) circle (2pt) node[above right] {$\mathcal{A}(\phi(x))$};
			
			\draw[->, ultra thick, gray!50] (Ax) -- (Aphi) node[midway, above, sloped, black] {Total Displacement};
			
			\draw[->, thick, vertcol] (Ax) -- (Mid) node[midway, below] {$\mathbf{c}(\phi)$ (Global Flux)};
			
			\draw[->, thick, horizcol] (Mid) -- (Aphi) node[midway, right] {$\mathcal{F}(\phi)(x)$ (Fluctuation)};
			
			\node[align=center, font=\small] at (3,-0.5) {$\mathcal{A}(\phi(x)) - \mathcal{A}(x) = \mathbf{c}(\phi) + \mathcal{F}(\phi)(x)$};
			
			\draw[step=1, gray!20, thin] (-0.5,-0.5) grid (6,4);
		\end{tikzpicture}
		\caption{Decomposition of the Abel--Jacobi displacement into the constant Flux vector and the zero-average Harmonic Fluctuation.}
	\end{figure}
	
	\begin{remark}[Geometric interpretation of the decomposition]
		The decomposition of Lemma~\ref{lem:AJ_decomposition}
		\[
		\widetilde{\mathcal A}\circ\widetilde\phi
		-
		\widetilde{\mathcal A}
		=
		\widetilde{\mathbf c}(\phi)
		+
		\widetilde{\mathcal F}(\phi),
		\]
		separates the action of a symplectomorphism on the universal Jacobian into two
		geometrically distinct contributions. The vector $\widetilde{\mathbf c}(\phi)\in
		H^1(\Sigma_g;\mathbb R)$ is a flux-like invariant naturally associated with the
		Abel--Jacobi displacement of $\phi$. It measures the average translation of the
		lifted Abel--Jacobi map and is independent of the choice of lift and of the
		symplectic isotopy joining the identity to $\phi$.
		
		In contrast, $\widetilde{\mathcal F}(\phi)$ has zero average and measures the
		residual pointwise deformation after the global translation has been removed.
		Thus, every symplectomorphism admits a canonical decomposition into a rigid
		translation in Jacobian coordinates together with a purely oscillatory harmonic
		displacement.
	\end{remark}
	
	\begin{remark}[Philosophy of the harmonic decomposition]
		The decomposition of Lemma~\ref{lem:AJ_decomposition} plays a role analogous to
		a Hodge decomposition for symplectic isotopies. The flux component $\mathbf
		c(\phi)$ records the global cohomological displacement, whereas the harmonic
		displacement $\mathcal F(\phi)$ contains the remaining pointwise geometric
		information. The remainder of this paper is devoted to showing that many local
		and global properties of symplectomorphisms---including norm estimates, local
		coordinates, and fixed-point questions---can be expressed directly in terms of
		the harmonic displacement field $\mathcal F(\phi)$.
	\end{remark}
	
	\begin{proposition}[Characterization of lifts of the flux class]
		\label{prop:lift_characterization}
		Let $(\Sigma_g,\omega)$ be a closed oriented surface of genus $g\ge2$, and let
		\[
		\widetilde{\mathbf c}:G_\omega(\Sigma_g)\longrightarrow
		\mathbb R^{2g}
		\simeq H^1(\Sigma_g;\mathbb R),
		\]
		be the map defined above. Then the following are equivalent for $\phi\in
		G_\omega(\Sigma_g)$:
		
		\begin{enumerate}
			\item[(i)] $\widetilde{\mathbf c}(\phi)=0$.
			
			\item[(ii)] The flux of $\phi$ is zero.
			
			\item[(iii)] $\phi\in\Ham(\Sigma_g,\omega)$.
		\end{enumerate}
	\end{proposition}
	
	\begin{proof}
		The implication $(i)\Longrightarrow(ii)$ is immediate since $\widetilde{\mathbf
			c}(\phi)$ is, by construction, a vector whose components are projections of the
		symplectic flux of any symplectic isotopy from the identity to $\phi$ onto the
		$L^2$-orthonormal basis $\{h_1,\dots,h_{2g}\}$. Since the flux group of
		$\Sigma_g$ is trivial for $g\ge2$, Banyaga's flux exact sequence becomes
		\[
		0
		\longrightarrow
		\Ham(\Sigma_g,\omega)
		\longrightarrow
		G_\omega(\Sigma_g)
		\xrightarrow{\operatorname{Flux}}
		H^1(\Sigma_g;\mathbb R)
		\longrightarrow
		0.
		\]
		Therefore, $(ii)\Longleftrightarrow(iii).$ Finally, if $\phi\in\Ham(\Sigma_g,\omega)$,
		then the flux class of $\phi$ vanishes. The implication $(iii)\Longrightarrow(i)$
		follows from Lemma \ref{lem:AJ_decomposition}-(3).
	\end{proof}
	
	\section{The Harmonic Flux Norm and its Properties}
	
	The decomposition of Lemma~\ref{lem:AJ_decomposition} naturally suggests
	measuring only the oscillatory component of the Abel--Jacobi displacement.
	
	\begin{proposition}\label{DECK}
		For every $\phi\in G_\omega(\Sigma_g)$, the function
		\[
		\widetilde{\mathcal F}(\phi): \widetilde{\Sigma}_g\longrightarrow\mathbb R^{2g},
		\]
		is invariant under the action of the deck transformation group $\Gamma$.
		Consequently, there exists a unique continuous map
		\[
		\mathcal F(\phi):\Sigma_g\longrightarrow\mathbb R^{2g},
		\]
		such that $\widetilde{\mathcal F}(\phi) = \mathcal F(\phi)\circ\pi.$ 
	\end{proposition}
	
	\begin{proof}
		Recall that
		\[
		\widetilde{\mathcal F}(\phi)(\tilde x)
		=
		\widetilde{\mathcal A}(\widetilde\phi(\tilde x))
		-
		\widetilde{\mathcal A}(\tilde x)
		-
		\widetilde{\mathbf c}(\phi).
		\]
		
		Let $\gamma$ be a deck transformation of the universal covering $\pi:\widetilde{\Sigma}_g\longrightarrow\Sigma_g.$ Since $\widetilde\phi$ is a lift of the map $\phi$, it commutes with every deck transformation: $\widetilde\phi\circ\gamma = \gamma\circ\widetilde\phi.$ Moreover, the lifted Abel--Jacobi map satisfies the equivariance relation $\widetilde{\mathcal A}(\gamma\tilde x) = \widetilde{\mathcal A}(\tilde x) + \gamma,$ where we identify the deck transformation $\gamma$ with the corresponding period vector in the lattice $\Gamma\subset\mathbb R^{2g}$. Therefore,
		\[
		\begin{aligned}
			\widetilde{\mathcal F}(\phi)(\gamma\tilde x)
			&=
			\widetilde{\mathcal A}
			\bigl(
			\widetilde\phi(\gamma\tilde x)
			\bigr)
			-
			\widetilde{\mathcal A}(\gamma\tilde x)
			-
			\widetilde{\mathbf c}(\phi)\\
			&=
			\widetilde{\mathcal A}
			\bigl(
			\gamma\widetilde\phi(\tilde x)
			\bigr)
			-
			\widetilde{\mathcal A}(\gamma\tilde x)
			-
			\widetilde{\mathbf c}(\phi)\\
			&=
			\left(
			\widetilde{\mathcal A}(\widetilde\phi(\tilde x))
			+\gamma
			\right)
			-
			\left(
			\widetilde{\mathcal A}(\tilde x)
			+\gamma
			\right)
			-
			\widetilde{\mathbf c}(\phi)\\
			&=
			\widetilde{\mathcal A}(\widetilde\phi(\tilde x))
			-
			\widetilde{\mathcal A}(\tilde x)
			-
			\widetilde{\mathbf c}(\phi)\\
			&=
			\widetilde{\mathcal F}(\phi)(\tilde x).
		\end{aligned}
		\]
		
		Hence $\widetilde{\mathcal F}(\phi)$ is invariant under every deck
		transformation. A continuous function on the universal cover that is invariant
		under the deck transformation group descends uniquely to the quotient. Therefore
		there exists a unique continuous map $\mathcal F(\phi):\Sigma_g\longrightarrow
		\mathbb R^{2g},$ satisfying $\widetilde{\mathcal F}(\phi) = \mathcal
		F(\phi)\circ\pi.$ Equivalently, $\mathcal F(\phi)(x) = \widetilde{\mathcal
			F}(\phi)(\tilde x),$ where $\tilde x$ is any lift of $x$.
	\end{proof}
	
	Proposition \ref{prop:lift_characterization} motivated the following definition:
	\begin{equation}
		\|\phi\|_{\mathrm{harm}} := \sup_{x \in \Sigma_g} \|\mathcal{F}(\phi)(x)\|_{\mathbb{R}^{2g}}.
	\end{equation}
	
	\begin{remark}[Equivalence of definitions]
		\label{rem:norm_equivalence}
		Because the components of $\mathcal{F}(\phi)$ are exactly the evaluations
		$\Delta(\phi, h_i)_x$ for the orthonormal basis $\{h_i\}$, the Euclidean norm of
		the vector $\mathcal{F}(\phi)(x)$ equals the supremum of $|\Delta(\phi,\alpha)_x|$
		over all unit-norm harmonic forms $\alpha$. Hence this definition coincides with
		the operator norm formulation given in the Introduction.
	\end{remark}
	
	\begin{proposition}[First-order expansion of the harmonic norm]
		\label{prop:first_order_harmonic_norm}
		Let $(\Sigma_g,\omega)$ be a closed symplectic surface of genus $g\ge2$, and let
		$\phi_\varepsilon=\exp(\varepsilon X_H)$ for $0<\varepsilon\ll1$ be the
		time-$\varepsilon$ map of the Hamiltonian vector field $X_H$ generated by a
		smooth Hamiltonian $H\in C^\infty(\Sigma_g)$. Then, the fluctuation vector
		$\mathcal F(\phi_\varepsilon): \Sigma_g \to \mathbb{R}^{2g}$ admits the uniform
		first-order expansion
		\[
		\mathcal F(\phi_\varepsilon)(x)
		=
		\varepsilon \big( h_1(X_H)(x), \dots, h_{2g}(X_H)(x) \big)
		+
		O(\varepsilon^2).
		\]
		Equivalently, for every harmonic $1$-form $h$, the scalar component
		$\mathcal{F}_h(\phi_\varepsilon)$ satisfies
		\[
		\mathcal F_h(\phi_\varepsilon)
		=
		\varepsilon\,h(X_H)
		+
		O(\varepsilon^2)
		\]
		in the $C^\infty$ topology. Consequently,
		\[
		\|\phi_\varepsilon\|_{\mathrm{harm}}
		=
		\varepsilon
		\sup_{\|h\|_{L^2}=1}
		\|h(X_H)\|_{C^0}
		+
		O(\varepsilon^2).
		\]
		
		In particular, if $H$ is nonconstant, then $
		\|\phi_\varepsilon\|_{\mathrm{harm}}>0,$ 
		for every sufficiently small $\varepsilon>0$.
	\end{proposition}
	
	\begin{proof}
		Since $\phi_\varepsilon$ is generated by the Hamiltonian vector field $X_H$, its
		pullback on differential forms admits the expansion $\phi_\varepsilon^*h = h +
		\varepsilon L_{X_H}h + O(\varepsilon^2),$ where $L_{X_H}$ denotes the Lie
		derivative. Because $h$ is harmonic, it is closed, so Cartan's formula gives
		$L_{X_H}h = d(\iota_{X_H}h) + \iota_{X_H}(dh) = d(h(X_H)).$ Hence,
		$\phi_\varepsilon^*h-h = d(\varepsilon\,h(X_H)) + O(\varepsilon^2).$ By
		definition, the harmonic primitive $\mathcal F_h(\phi_\varepsilon)$ is the
		unique zero-average function satisfying $d\mathcal F_h(\phi_\varepsilon) =
		\phi_\varepsilon^*h-h.$
		
		Notice that the function $h(X_H)$ already has zero average with respect to the
		symplectic volume form $\omega$. Indeed, since $h \wedge \omega$ is a $3$-form
		on a surface, it vanishes, and therefore
		\[
		0 = \iota_{X_H}(h \wedge \omega) = (\iota_{X_H}h)\omega - h \wedge
		(\iota_{X_H}\omega) = h(X_H)\omega - h \wedge dH.
		\]
		Because $h$ is closed, $h \wedge dH = -d(Hh)$, which is exact. By Stokes'
		theorem, its integral over $\Sigma_g$ vanishes. Since both $\mathcal
		F_h(\phi_\varepsilon)$ and $h(X_H)$ have zero average, we conclude that
		\[
		\mathcal F_h(\phi_\varepsilon)
		=
		\varepsilon\,h(X_H)
		+
		O(\varepsilon^2).
		\]
		
		Because the space of harmonic $1$-forms of unit $L^2$-norm is a compact,
		finite-dimensional sphere, the $O(\varepsilon^2)$ error term is uniform in $h$.
		Taking the supremum over all such harmonic forms yields
		\[
		\|\phi_\varepsilon\|_{\mathrm{harm}}
		=
		\varepsilon
		\sup_{\|h\|_{L^2}=1}
		\|h(X_H)\|_{C^0}
		+
		O(\varepsilon^2).
		\]
		
		Finally, suppose that $\sup_{\|h\|_{L^2}=1}\|h(X_H)\|_{C^0}=0.$ Then every
		harmonic $1$-form vanishes on $X_H$. However, for surfaces of genus $g \ge 2$, a
		fundamental theorem (tied to the Abel--Jacobi map being an embedding) states
		that harmonic $1$-forms have no common zeros. More precisely, they separate
		tangent vectors: at any point $p$, if we have a non-zero vector $v \in
		T_p\Sigma_g$, there is at least one harmonic $1$-form $h$ such that $h(v) \neq
		0$. Since $H$ is nonconstant, $X_H$ is non-zero at some point $p$, so there
		exists some unit-norm harmonic form $h$ such that $h(X_H)(p) \neq 0$. This
		contradicts the assumption. Hence, 
		\[
		C := \sup_{\|h\|_{L^2}=1}\|h(X_H)\|_{C^0}>0.
		\]
		Because $\|\phi_\varepsilon\|_{\mathrm{harm}} = C\varepsilon + O(\varepsilon^2)$,
		the strictly positive linear term $C\varepsilon$ dominates the $O(\varepsilon^2)$
		error for small $\varepsilon$. This proves that $\|\phi_\varepsilon\|_{\mathrm{harm}}>0$
		for all sufficiently small $\varepsilon>0$.
	\end{proof}
	
	\begin{example}[A localized Hamiltonian twist on a genus-two surface]
		\label{ex:genus2_twist}
		To illustrate the geometric meaning of the harmonic norm, we consider a
		localized Hamiltonian flow on a closed surface $(\Sigma_2,\omega)$ of genus two.
		Choose an embedded cylindrical neighborhood $C\simeq S^1\times[-w,w],$ with
		coordinates $(\theta,y)\in\mathbb R/\mathbb Z\times[-w,w],$ and symplectic form
		$\omega=d\theta\wedge dy.$ Geometrically, one may regard $C$ as a narrow neck
		joining the two handles of $\Sigma_2$. Let
		\[
		H(\theta,y)=
		\begin{cases}
			\cos^2\!\bigl(\tfrac{\pi y}{2w}\bigr), & |y|\le w,\\[4pt]
			0, & \text{outside }C,
		\end{cases}
		\]
		smoothed near the boundary so that $H\in C^\infty(\Sigma_2)$. Since
		$\iota_{X_H}\omega=dH,$ the Hamiltonian vector field is
		\[
		X_H
		=
		H'(y)\frac{\partial}{\partial\theta}
		=
		-\frac{\pi}{2w}
		\sin\!\left(\frac{\pi y}{w}\right)
		\frac{\partial}{\partial\theta}.
		\]
		
		Thus the Hamiltonian flow generates a shear (or twist) along the neck. Its speed
		vanishes at $y=0,$ and $ y=\pm w,$ and reaches its maximum magnitude along the
		circles $y=\pm\frac{w}{2}.$ Let $\phi_\varepsilon = \exp(\varepsilon X_H)$
		denote the time-$\varepsilon$ map. By
		Proposition~\ref{prop:first_order_harmonic_norm},
		\[
		\|\phi_\varepsilon\|_{\mathrm{harm}}
		=
		\varepsilon
		\sup_{\|h\|_{L^2}=1}
		\|h(X_H)\|_{C^0}
		+
		O(\varepsilon^2).
		\]
		Assume now that the Bergman metric associated with the harmonic $1$-forms has
		been introduced. Since the Bergman norm is characterized by $\|v\|_B =
		\sup_{\|h\|_{L^2}=1} |h(v)|,$ we obtain
		\[
		\|\phi_\varepsilon\|_{\mathrm{harm}}
		=
		\varepsilon
		\sup_{p\in\Sigma_2}
		\|X_H(p)\|_B
		+
		O(\varepsilon^2).
		\]
		
		Substituting the expression for $X_H$ gives
		\[
		\|\phi_\varepsilon\|_{\mathrm{harm}}
		=
		\varepsilon
		\frac{\pi}{2w}
		\sup_{|y|\le w}
		\left(
		\left|
		\sin\!\left(\frac{\pi y}{w}\right)
		\right|
		\left\|
		\frac{\partial}{\partial\theta}
		\right\|_B
		\right)
		+
		O(\varepsilon^2).
		\]
		
		If the neck is sufficiently thin and nearly rotationally symmetric, the Bergman
		length of $\partial_\theta$ varies only slightly along the cylinder. Writing
		$B_{\mathrm{neck}} = \sup_{|y|\le w} \left\|\frac{\partial}{\partial\theta}
		\right\|_B,$ the preceding formula simplifies to 
		\[
		\|\phi_\varepsilon\|_{\mathrm{harm}} = \varepsilon \frac{\pi}{2w} B_{\mathrm{neck}}
		+ O(\varepsilon^2).
		\]
		This expression exhibits a natural separation between two geometric
		contributions. The factor $\frac{\pi}{2w}$ depends only on the chosen
		Hamiltonian and measures the maximal shear generated by the flow, while the
		factor $B_{\mathrm{neck}}$ depends only on the global geometry of the Riemann
		surface through the Bergman metric induced by harmonic $1$-forms. It measures
		how strongly the harmonic geometry detects motion along the neck.
		
		Consequently, the first-order harmonic norm separates the local dynamics of the
		Hamiltonian flow from the global Hodge geometry of the underlying surface. This
		illustrates the philosophy developed throughout the paper: the harmonic norm
		simultaneously captures dynamical information and the geometry encoded by the
		Abel-Jacobi embedding.
	\end{example}

	\begin{remark}[Generalizations]
		The construction in Example~\ref{ex:genus2_twist} extends naturally to:
		\begin{enumerate}
			\item Surfaces of arbitrary genus $g \ge 2$ by placing the 
			Hamiltonian support in any cylindrical neck region,
			\item Multiple simultaneous twists on different necks, yielding 
			non-trivial interactions between local dynamics,
			\item Time-dependent Hamiltonians that create more complex shear 
			patterns.
		\end{enumerate}
		In all cases, the first-order formula continues to separate dynamical 
		and geometric contributions, providing explicit computability.
	\end{remark}

	\begin{proposition}[Kernel of the Harmonic Norm]
		\label{prop:kernel}
		Let $\phi\in G_\omega(\Sigma_g)$ and let
		\[
		\widetilde{\mathcal A}(\widetilde\phi(\tilde x))
		-
		\widetilde{\mathcal A}(\tilde x)
		=
		\widetilde{\mathbf c}
		+
		\widetilde{\mathcal F}(\phi)(\tilde x),
		\]
		be the decomposition of Lemma~\ref{lem:AJ_decomposition}, where
		$\widetilde{\mathbf c}\in H^1(\Sigma_g;\mathbb R)$ is a representative of the
		flux class and $\widetilde{\mathcal F}(\phi)$ has zero mean. Then the following
		are equivalent:
		
		\begin{enumerate}
			\item $\|\phi\|_{\mathrm{harm}}=0$;
			
			\item
			\[
			\widetilde{\mathcal A}(\widetilde\phi(\tilde x))
			-
			\widetilde{\mathcal A}(\tilde x)
			=
			\widetilde{\mathbf c}
			\qquad
			\forall\,\tilde x\in\widetilde{\Sigma}_g;
			\]
			
			\item
			\[
			\phi^*\alpha=\alpha
			\qquad
			\forall\,\alpha\in\mathcal H^1(\Sigma_g).
			\]
		\end{enumerate}
	\end{proposition}
	
	\begin{proof}
		Suppose first that $\|\phi\|_{\mathrm{harm}}=0.$ By definition, $\widetilde{\mathcal F}
		(\phi)\equiv0.$ Hence, 
		\[
		\widetilde{\mathcal A}(\widetilde\phi(\tilde x)) - \widetilde{\mathcal A}(\tilde x)
		= \widetilde{\mathbf c},
		\]
		which proves (ii). Assume now that (ii) holds. Differentiating both sides gives
		$d(\widetilde{\mathcal A}\circ\widetilde\phi) = d\widetilde{\mathcal A}.$
		Since $d\widetilde{\mathcal A} = (\pi^*h_1,\ldots,\pi^*h_{2g}),$ we obtain
		\[
		\widetilde\phi^{\,*}(\pi^*h_i)
		=
		\pi^*h_i,
		\qquad
		i=1,\ldots,2g.
		\]
		Because pullback commutes with the covering projection, $\pi^*(\phi^*h_i) =
		\pi^*h_i.$ As $\pi^*$ is injective, $\phi^*h_i=h_i,$ for $i=1,\ldots,2g.$ By
		linearity, $\phi^*\alpha=\alpha$ for every harmonic $1$-form $\alpha$, proving
		(iii). Finally, assume (iii). Then, $\phi^*h_i=h_i$ for $i=1,\ldots,2g,$ so
		$d(\widetilde{\mathcal A}\circ\widetilde\phi) = d\widetilde{\mathcal A}.$
		Therefore, $d\left(\widetilde{\mathcal A}\circ\widetilde\phi -
		\widetilde{\mathcal A}\right) =0.$ Since the universal cover
		$\widetilde{\Sigma}_g$ is connected, the function $\widetilde{\mathcal
			A}(\widetilde\phi(\tilde x)) - \widetilde{\mathcal A}(\tilde x)$ is constant.
		By the uniqueness in the decomposition of Lemma~\ref{lem:AJ_decomposition}, this
		constant is precisely $\widetilde{\mathbf c}$, so $\widetilde{\mathcal
			F}(\phi)\equiv0.$ Hence, $\|\phi\|_{\mathrm{harm}}=0.$ Thus (i), (ii), and (iii)
		are equivalent.
	\end{proof}
	
	We now prove the two metric lemmas that underpin all subsequent estimates.
	
	\begin{lemma}[Local Rigidity of the Abel--Jacobi Map]
		\label{lem:rigidity}
		Let $\Sigma_g$ be a closed, oriented surface of genus $g \ge 2$, endowed with a
		Riemannian metric $\mathsf g$. There exist constants $c_1 > 0$ and $r > 0$ such
		that, whenever $d_{\mathsf g}(x,y) < r$, their images under the Abel--Jacobi map
		are separated by at least a fixed linear multiple of their surface distance:
		\[
		\|\mathcal A(x)-\mathcal A(y)\|_{\mathbb R^{2g}} \;\ge\; c_1 \cdot
		d_{\mathsf g}(x,y).
		\]
	\end{lemma}
	
	\begin{proof}
		Because the Abel--Jacobi map $\mathcal A:\Sigma_g \hookrightarrow
		\operatorname{Jac}(\Sigma_g)$ is an immersion (in fact, a holomorphic embedding
		for a suitable complex structure), its differential $d\mathcal A_p : T_p\Sigma_g
		\longrightarrow \mathbb R^{2g}$ is injective at every $p\in\Sigma_g$. For each
		$p$, define the \emph{minimal stretching} of $d\mathcal A$ on the unit tangent
		sphere by
		\[
		m(p) \;=\; \min\bigl\{ \|d\mathcal A_p(v)\|_{\mathbb R^{2g}} \;:\; v\in
		T_p\Sigma_g,\; \|v\|_{\mathsf g}=1 \bigr\}.
		\]
		Injectivity implies $m(p)>0$. The unit tangent bundle of $\Sigma_g$ is compact,
		and the function $(p,v)\mapsto \|d\mathcal A_p(v)\|$ is continuous, so $p\mapsto
		m(p)$ is continuous as well. Hence, for any given $p$, we can choose a
		neighborhood $U_p$ of $p$ and a constant $c_p>0$ such that
		\begin{equation}\label{eq:linear-stretch}
			\|d\mathcal A_q(v)\| \;\ge\; c_p \|v\|_{\mathsf g}
			\qquad \forall q\in U_p,\; \forall v\in T_q\Sigma_g.
		\end{equation}
		For instance, take $c_p = \frac12 m(p)$ and $U_p$ a set where $m(q) \ge \frac12
		m(p)$; such a $U_p$ exists by continuity of $m$.	The existence of such neighborhoods follows from the exponential map and the
		fact that normal coordinates give convex balls for small enough radii. Shrink
		$U_p$ further so that it is geodesically convex: for any $x,y\in U_p$ there
		exists a unique minimizing geodesic $\gamma:[0,1]\to\Sigma_g$ from $x$ to $y$
		whose image lies entirely in $U_p$. Now let $x,y\in U_p$ and let $\gamma$ be
		this geodesic. Then
		\[
		\mathcal A(y)-\mathcal A(x) \;=\; \int_0^1 d\mathcal A_{\gamma(t)}\bigl(\dot\gamma(t)\bigr)\,dt.
		\]
		Taking norms and applying \eqref{eq:linear-stretch} along the curve,
		\[
		\begin{aligned}
			\|\mathcal A(y)-\mathcal A(x)\|_{\mathbb R^{2g}}
			&\;\ge\; \int_0^1 \bigl\| d\mathcal A_{\gamma(t)}(\dot\gamma(t)) \bigr\|\,dt \\
			&\;\ge\; c_p \int_0^1 \|\dot\gamma(t)\|_{\mathsf g}\,dt \\
			&\;=\; c_p\, d_{\mathsf g}(x,y).
		\end{aligned}
		\]
		Thus we have exhibited, for each $p$, a neighborhood $V_p:=U_p$ (already
		geodesically convex) and a constant $c_p>0$ satisfying
		\[
		\|\mathcal A(x)-\mathcal A(y)\| \;\ge\; c_p\, d_{\mathsf g}(x,y) \qquad
		\forall x,y\in V_p.
		\]
		
		The family $\{V_p\}_{p\in\Sigma_g}$ forms an open cover of the compact surface
		$\Sigma_g$. Choose a finite subcover $\{V_{p_1},\dots,V_{p_N}\}$ and let $r>0$
		be a Lebesgue number of this covering (every subset of diameter $<r$ is
		contained in some $V_{p_i}$). Define
		\[
		c_1 \;=\; \min\{ c_{p_1}, \dots, c_{p_N} \} \;>\;0.
		\]
		Finally, if $x,y\in\Sigma_g$ satisfy $d_{\mathsf g}(x,y)<r$, then the set
		$\{x,y\}$ has diameter $<r$, so it lies entirely within one of the $V_{p_i}$.
		Consequently,
		\[
		\|\mathcal A(x)-\mathcal A(y)\| \;\ge\; c_{p_i}\, d_{\mathsf g}(x,y) \;\ge\;
		c_1\, d_{\mathsf g}(x,y),
		\]
		which completes the proof.
	\end{proof}
	
	\begin{lemma}[Global Separation]
		\label{lem:global_sep}
		Let $r>0$ be as in Lemma~\ref{lem:rigidity}. Since $\mathcal A:\Sigma_g
		\hookrightarrow \operatorname{Jac}(\Sigma_g)$ is an embedding and $\Sigma_g$ is
		compact, the continuous function $(x,y)\mapsto \|\mathcal A(x)-\mathcal A(y)\|$
		attains a positive minimum on the compact set $\{(x,y): d_{\mathsf g}(x,y)\ge
		r\}$. Hence there exists $\delta>0$ such that
		\begin{equation}
			d_{\mathsf g}(x,y)\ge r \implies \|\mathcal A(x)-\mathcal A(y)\|\ge \delta.
			\label{eq:global_sep}
		\end{equation}
	\end{lemma}
	
	Combining the two lemmas, we have a uniform lower bound for all distances: for
	all $x,y\in \Sigma_g$,
	\begin{equation}\label{Lower-bound}
		\|\mathcal A(x)-\mathcal A(y)\| \ge \min\{ c_1\, d_{\mathsf g}(x,y),\;
		\delta \}.
	\end{equation}
	
	\begin{theorem}[Non-degeneracy on Hamiltonian]
		\label{thm:non_deg_ham}
		Let $(\Sigma_g,\omega)$ be a closed oriented surface of genus $g\ge2$. Then the
		harmonic norm is non-degenerate on $\Ham(\Sigma_g,\omega)$. 
	\end{theorem}
	
	\begin{proof}
		Assume that $\|\phi\|_{\mathrm{harm}} = 0.$ By Proposition~\ref{prop:kernel},
		\[
		\mathcal A(\phi(x))-\mathcal A(x)
		=
		\widetilde{\mathbf c}(\phi)\bmod\Gamma,
		\]
		where $\widetilde{\mathbf c}(\phi)\in \mathbb R^{2g} \simeq
		H^1(\Sigma_g;\mathbb R)$ is the unique representative of the flux class. Since
		$\phi$ is Hamiltonian, then $\widetilde{\mathbf c}(\phi) = 0$ by Lemma
		\ref{lem:AJ_decomposition}-(3). Hence, $\mathcal A(\phi(x)) = \mathcal A(x)$
		for all $x\in\Sigma_g.$ Finally, since $\mathcal A$ is injective,
		$\phi(x)=x$ for $x\in\Sigma_g,$ and therefore $\phi=\mathrm{id}.$ The converse
		implication is immediate from the definition of $\|\cdot\|_{\mathrm{harm}}$.
	\end{proof}
	
	\begin{proposition}[Rigidity of the Abel--Jacobi image]
		\label{prop:aj_rigidity}
		Let $\Sigma_g$ be a closed Riemann surface of genus $g\ge2$, and let $\mathcal
		A:\Sigma_g\longrightarrow \operatorname{Jac}(\Sigma_g)$ be its Abel--Jacobi
		embedding. Suppose that $\mathcal A(\Sigma_g)+a=\mathcal A(\Sigma_g)$ for some
		$a\in\operatorname{Jac}(\Sigma_g).$ Then, $a=0.$ Equivalently, the Abel--Jacobi
		image admits no non-trivial translation symmetry.
	\end{proposition}
	
	\begin{proof}
		Set $C = \mathcal A(\Sigma_g)\subset \operatorname{Jac}(\Sigma_g)$.
		Assume $C + a = C$ and define the translation map
		\[
		T_a : \operatorname{Jac}(\Sigma_g)\longrightarrow \operatorname{Jac}(\Sigma_g),
		\qquad T_a(x)=x+a.
		\]
		By hypothesis $T_a(C)=C$, so the restriction $T_a|_C$ is a biholomorphic
		automorphism of the compact Riemann surface $C$. Via the Abel--Jacobi embedding
		$\mathcal A$, we transport this automorphism to $\Sigma_g$:
		\[
		\tau := \mathcal A^{-1} \circ T_a \circ \mathcal A : \Sigma_g \longrightarrow
		\Sigma_g,
		\]
		which is a holomorphic automorphism of $\Sigma_g$.
		
		\medskip\noindent
		\textbf{Action on homology.}
		The translation $T_a$ is homotopic to the identity on the torus (the homotopy
		$t\mapsto T_{ta}$ connects $T_0=\operatorname{id}$ to $T_a$). Hence it induces
		the identity on the singular homology group $H_1(\operatorname{Jac}(\Sigma_g);\mathbb{Z})$:
		\[
		(T_a)_* = \operatorname{id}_{H_1(\operatorname{Jac}(\Sigma_g))}.
		\]
		It is a classical theorem that the Abel--Jacobi embedding induces an isomorphism
		\[
		\mathcal A_*: H_1(\Sigma_g;\mathbb Z) \stackrel{\cong}{\longrightarrow}
		H_1(\operatorname{Jac}(\Sigma_g);\mathbb Z).
		\]
		Consequently,
		\[
		\tau_* = (\mathcal A_*)^{-1} \circ (T_a)_* \circ \mathcal A_* =
		\operatorname{id}_{H_1(\Sigma_g;\mathbb{Z})}.
		\]
		
		\medskip\noindent
		\textbf{Lefschetz number.}
		Consider the Lefschetz number of the continuous map $\tau$,
		\[
		L(\tau) = \sum_{i=0}^{2} (-1)^i \operatorname{Tr}\bigl(\tau_*|_{H_i(\Sigma_g;\mathbb{Q})}\bigr).
		\]
		Since $\tau$ is biholomorphic, it preserves the complex orientation.
		Consequently, $\tau_*=\mathrm{id}$ on $H_0(\Sigma_g;\mathbb Q)$ and
		$H_2(\Sigma_g;\mathbb Q).$ Thus, $\operatorname{Tr}(\tau_*|_{H_0}) = 1$ and
		$\operatorname{Tr}(\tau_*|_{H_2}) = 1$. On $H_1$ we have just shown $\tau_* =
		\operatorname{id}$, so $\operatorname{Tr}(\tau_*|_{H_1}) = 2g$. Therefore,
		\[
		L(\tau) = 1 - 2g + 1 = 2 - 2g.
		\]
		Since $g \ge 2$, we have $L(\tau) \neq 0$.
		
		\medskip\noindent
		\textbf{No fixed points.}
		If $a \neq 0$, then the translation $T_a$ has no fixed points on the Jacobian
		(the equation $x+a=x$ has no solution in an abelian variety). Now suppose, for a
		contradiction, that $\tau(x) = x$ for some $x\in\Sigma_g$. Then
		\[
		T_a(\mathcal A(x)) = \mathcal A(\tau(x)) = \mathcal A(x),
		\]
		which contradicts the fact that a nontrivial translation on an abelian variety
		has no fixed points. Hence $\tau$ is fixed-point-free.
		
		A continuous map on a compact manifold without fixed points must have Lefschetz
		number zero by the Lefschetz fixed-point theorem. But we computed $L(\tau) =
		2-2g \neq 0$. This contradiction shows that no nontrivial translation preserves
		$\mathcal A(\Sigma_g)$. Hence $a=0$.
	\end{proof}
	
	\begin{remark}
		The proposition is a special case of the general fact that the stabilizer of the
		Abel--Jacobi image of a curve of genus at least two inside its Jacobian is
		trivial. We included the above proof because it is elementary and relies only on
		homological properties of the Abel--Jacobi embedding and the Lefschetz
		fixed-point theorem.
	\end{remark}
	
	\begin{theorem}[Non-degeneracy on the Identity Component]
		\label{thm:non_deg_full}
		Let $(\Sigma_g,\omega)$ be a closed oriented surface of genus $g\ge2$. Then the
		harmonic norm is non-degenerate:
		\[
		\|\phi\|_{\mathrm{harm}}=0
		\quad\Longleftrightarrow\quad
		\phi=\mathrm{id}.
		\]
	\end{theorem}
	
	\begin{proof}
		Assume that $\|\phi\|_{\mathrm{harm}}=0.$ By Proposition~\ref{prop:kernel}, the
		Abel--Jacobi displacement is constant:
		\[
		\mathcal A(\phi(x))-\mathcal A(x) = q(\widetilde{\mathbf c}(\phi)),
		\]
		for every $x\in\Sigma_g$, where
		\[
		q:H^1(\Sigma_g;\mathbb R)\longrightarrow H^1(\Sigma_g;\mathbb R)/\Gamma =
		\operatorname{Jac}(\Sigma_g),
		\]
		is the quotient map. Since $\phi$ is a diffeomorphism, $\phi(\Sigma_g)=\Sigma_g.$
		Therefore,
		\[
		\mathcal A(\Sigma_g) = \mathcal A(\phi(\Sigma_g)) = \mathcal A(\Sigma_g) +
		q(\widetilde{\mathbf c}(\phi)).
		\]
		Hence the Abel--Jacobi image is invariant under the translation
		$T_{q(\widetilde{\mathbf c}(\phi))}.$ By Proposition~\ref{prop:aj_rigidity}, the
		Abel--Jacobi image admits no non-trivial translation symmetry. Consequently,
		$q(\widetilde{\mathbf c}(\phi))=0$ in $\operatorname{Jac}(\Sigma_g).$	It follows that $\mathcal A(\phi(x)) = \mathcal A(x)$ for all $x\in\Sigma_g.$
		Since the Abel--Jacobi map is an embedding for $g\ge2$, it is injective.
		Therefore, $\phi(x)=x$ for all $x\in\Sigma_g,$ and hence $\phi=\mathrm{id}.$
		Conversely, if $\phi=\mathrm{id}$, then every harmonic displacement vanishes and
		therefore $\|\phi\|_{\mathrm{harm}}=0.$ 
	\end{proof}
	
	\begin{lemma}[Local Stability of Lifts]
		\label{lem:stability_lifts}
		Let $\pi:\widetilde{M}\to M$ be a regular covering of a compact Riemannian
		manifold $M$ equipped with the pullback metric. Let $K\subset\widetilde M$ be a
		compact connected subset, and let $\tilde{x}_0\in K$ be a fixed basepoint.
		Suppose that $f_n,f:M\to M$ are continuous maps satisfying $f_n\to f$ uniformly.
		Let $\widetilde f_n,\widetilde f:\widetilde M\to\widetilde M$ be lifts of
		$f_n\circ\pi$ and $f\circ\pi$, respectively, satisfying $\widetilde
		f_n(\tilde{x}_0)\longrightarrow \widetilde f(\tilde{x}_0).$ Then $\widetilde
		f_n|_K\longrightarrow \widetilde f|_K$ uniformly on $K$.
	\end{lemma}
	
	\begin{proof}
		Since $K$ is compact and $\widetilde f$ is continuous, the set $K_1:=\widetilde
		f(K)$ is a compact subset of $\widetilde M$. Because $\pi$ is a covering local
		isometry, every point $\tilde y\in K_1$ admits a radius $r_{\tilde y}>0$ such
		that the ball $B_{\widetilde M}(\tilde y,r_{\tilde y})$ is contained in an
		evenly covered neighborhood $\widetilde U_{\tilde y}\subset\widetilde M$ and is
		geodesically convex. The family $\left\{ B_{\widetilde M}(\tilde y,r_{\tilde
			y}/2) \right\}_{\tilde y\in K_1}$ is an open cover of the compact set $K_1$.
		Hence there exist finitely many points $\tilde y_1,\ldots,\tilde y_m\in K_1$
		such that
		\[
		K_1\subset \bigcup_{j=1}^{m} B_{\widetilde M} (\tilde y_j,r_j/2),
		\]
		where $r_j:=r_{\tilde y_j}.$ Define
		\[
		\delta_K = \frac14\min_{1\le j\le m}r_j.
		\]
		Then for every $\tilde y\in K_1$, there exists an index $j$ such that $\tilde
		y\in B_{\widetilde M}(\tilde y_j,r_j/2).$ Therefore, $B_{\widetilde M}(\tilde
		y,\delta_K) \subset B_{\widetilde M}(\tilde y_j,r_j) \subset \widetilde
		U_{\tilde y_j}.$ Hence every $\delta_K$-ball centered at a point of $K_1$ is
		contained in a geodesically convex evenly covered neighborhood. Since $f_n\longrightarrow
		f$ uniformly on $M$, there exists $N$ such that for every $n\ge N$,
		\[
		\sup_{x\in M} d_M(f_n(x),f(x)) < \delta_K.
		\]
		Fix $n\ge N$ and $\tilde x\in K$. Put $\tilde y=\widetilde f(\tilde x)\in K_1.$
		The points $f(\pi(\tilde x))$ and $f_n(\pi(\tilde x))$ are at distance less than
		$\delta_K$. Hence there exists a unique minimizing geodesic $\gamma_{n,\tilde
			x}:[0,1]\to M$ joining them. Lift this geodesic to a path $\widetilde\gamma_{n,\tilde
			x}$ starting at $\widetilde\gamma_{n,\tilde x}(0) = \widetilde f(\tilde x).$
		Let $\widetilde g_n(\tilde x) = \widetilde\gamma_{n,\tilde x}(1).$ By
		construction, $\pi(\widetilde g_n(\tilde x)) = f_n(\pi(\tilde x)),$ so
		$\widetilde g_n$ is a lift of $f_n\circ\pi|_K.$ We now estimate the distance
		between the endpoints. Since the length of the lifted geodesic equals the length
		of the original geodesic, we have
		\[
		\operatorname{Length} (\widetilde\gamma_{n,\tilde x}) =
		d_M(f_n(\pi(\tilde x)),f(\pi(\tilde x))) < \delta_K.
		\]
		Because $B_{\widetilde M}(\widetilde f(\tilde x),\delta_K)$ is contained in a
		geodesically convex evenly covered neighborhood, the lifted geodesic remains
		inside this neighborhood. On this neighborhood, $\pi$ is an isometry. Therefore,
		\[
		d_{\widetilde M} (\widetilde g_n(\tilde x),\widetilde f(\tilde x)) =
		d_M(f_n(\pi(\tilde x)),f(\pi(\tilde x))) \le \|f_n-f\|_{C^0}.
		\]
		Consequently,
		\[
		\sup_{\tilde x\in K} d_{\widetilde M} (\widetilde g_n(\tilde x),\widetilde
		f(\tilde x)) \le \|f_n-f\|_{C^0}.
		\tag{a}
		\]
		
		It remains to identify $\widetilde g_n$ with the prescribed lift $\widetilde
		f_n$. By assumption, $\widetilde f_n(\tilde x_0) \longrightarrow \widetilde
		f(\tilde x_0).$ Moreover, applying (a) at $\tilde x_0$ gives $\widetilde
		g_n(\tilde x_0) \longrightarrow \widetilde f(\tilde x_0).$ Hence for all
		sufficiently large $n$, both points $\widetilde f_n(\tilde x_0)$ and
		$\widetilde g_n(\tilde x_0)$ belong to the same evenly covered neighborhood of
		$\widetilde f(\tilde x_0)$. They are lifts of the same point:
		\[
		\pi(\widetilde f_n(\tilde x_0)) = f_n(\pi(\tilde x_0)) =
		\pi(\widetilde g_n(\tilde x_0)).
		\]
		Since $\pi$ is injective on this evenly covered neighborhood, we conclude that
		$\widetilde f_n(\tilde x_0) = \widetilde g_n(\tilde x_0).$ Therefore,
		$\widetilde f_n|_K$ and $\widetilde g_n$ are two lifts of the same map
		$f_n\circ\pi|_K$ which agree at the point $\tilde x_0$. Since $K$ is connected,
		the uniqueness of lifts implies $\widetilde f_n|_K = \widetilde g_n.$ Combining
		this equality with (a), we obtain
		\[
		\sup_{\tilde x\in K} d_{\widetilde M} (\widetilde f_n(\tilde x),\widetilde
		f(\tilde x)) \le \|f_n-f\|_{C^0}.
		\]
		Because $\|f_n-f\|_{C^0}\longrightarrow0,$ we conclude that $\widetilde f_n|_K
		\longrightarrow \widetilde f|_K$ uniformly on $K$.
	\end{proof}
	
	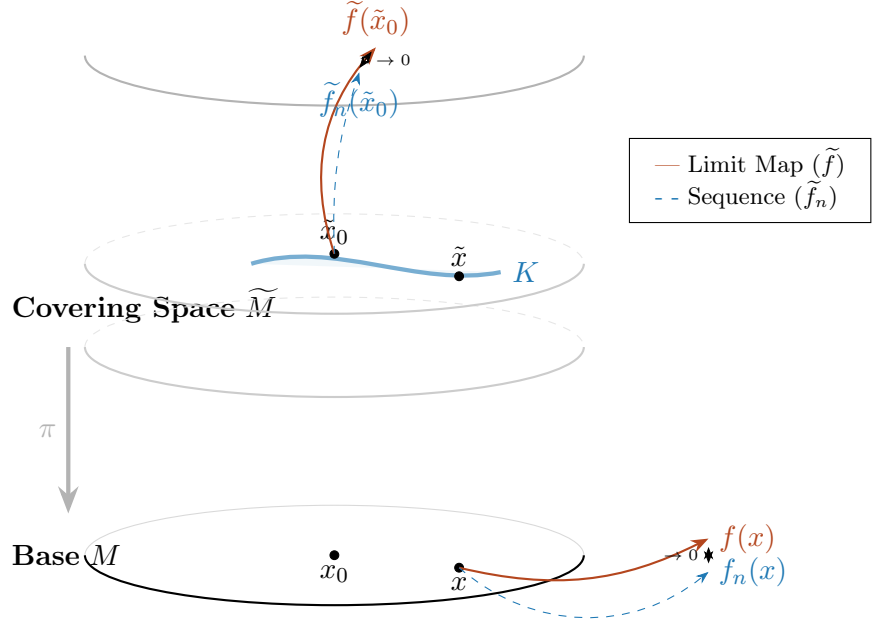
\begin{figure}[htbp]
		\centering
		\begin{tikzpicture}[scale=1.1, >=Stealth]
			\node[anchor=west] at (-4, 3) {\textbf{Covering Space} $\widetilde{M}$};
			
			\foreach \h in {2.5, 3.5} {
				\draw[gray!40, thick] (-3, \h) arc (180:360:3 and 0.6);
				\draw[gray!20, thin, dashed] (-3, \h) arc (180:0:3 and 0.6);
			}
			
			\draw[color=vertcol, fill=vertcol!10, ultra thick, opacity=0.6] 
			(-1, 3.5) .. controls (0, 3.8) and (1, 3.2) .. (2, 3.4) node[right, opacity=1] {$K$};
			
			\filldraw (0, 3.62) circle (1.5pt) node[above] {$\tilde{x}_0$};
			\filldraw (1.5, 3.35) circle (1.5pt) node[above] {$\tilde{x}$};
			
			\draw[gray!60, thick] (-3, 6) arc (180:360:3 and 0.6);
			
			\draw[->, color=horizcol, thick] (0, 3.62) to[bend left=30] (0.5, 6.1) node[above] {$\widetilde{f}(\tilde{x}_0)$};
			\draw[->, color=vertcol, dashed] (0, 3.62) to[bend left=10] (0.3, 5.8) node[below] {$\widetilde{f}_n(\tilde{x}_0)$};
			
			\draw[<->, thin] (0.45, 6.05) -- (0.3, 5.85) node[midway, right, font=\tiny] {$\to 0$};
			
			\draw[->, ultra thick, gray!60] (-3.2, 2.5) -- (-3.2, 0.5) node[midway, left] {$\pi$};
			
			\node[anchor=west] at (-4, 0) {\textbf{Base} $M$};
			\draw[thick] (-3, 0) arc (180:360:3 and 0.6);
			\draw[thin, gray!30] (-3, 0) arc (180:0:3 and 0.6);
			
			\filldraw (0, 0) circle (1.5pt) node[below] {$x_0$};
			\filldraw (1.5, -0.15) circle (1.5pt) node[below] {$x$};
			
			\draw[->, color=horizcol, thick] (1.5, -0.15) to[bend right=20] (4.5, 0.2) node[right] {$f(x)$};
			\draw[->, color=vertcol, dashed] (1.5, -0.15) to[bend right=40] (4.5, -0.2) node[right] {$f_n(x)$};
			
			\draw[<->, thin] (4.5, 0.1) -- (4.5, -0.1) node[midway, left, font=\tiny] {$\to 0$};
			
			\node[draw, font=\footnotesize, fill=white] at (5, 4.5) {
				\begin{tabular}{l}
					\textcolor{horizcol}{---} Limit Map ($\widetilde{f}$) \\
					\textcolor{vertcol}{- -} Sequence ($\widetilde{f}_n$)
				\end{tabular}
			};
			
		\end{tikzpicture}
		\tiny{	\caption{Stability of Lifts: Uniform convergence of $f_n \to f$ on the base, combined with convergence of the lifts at a single point $\tilde{x}_0$, forces uniform convergence of $\widetilde{f}_n \to \widetilde{f}$ on the connected set $K$.}}
	\end{figure}

	\begin{theorem}[$C^0$-Continuity of the Harmonic Norm]
		\label{thm:C0continuity}
		The harmonic norm $\|\cdot\|_{\mathrm{harm}}: G_\omega(\Sigma_g) \longrightarrow
		[0,\infty)$ is continuous with respect to the $C^0$-topology.
	\end{theorem}
	
	\begin{proof}
		Let $\phi_n\xrightarrow{C^0}\phi$ in $G_\omega(\Sigma_g)$. Fix a fundamental
		domain $D\subset\widetilde{\Sigma}_g$ with compact closure and choose a
		base-point $\tilde x_0\in D.$ Let $\widetilde\phi$ be a lift of $\phi$. For each
		sufficiently large $n$, choose the lift $\widetilde\phi_n$ normalized so that
		$\widetilde\phi_n(\tilde x_0) \longrightarrow \widetilde\phi(\tilde x_0).$ By
		Lemma~\ref{lem:stability_lifts}, $\widetilde\phi_n \longrightarrow
		\widetilde\phi$ uniformly on $\overline D$.	Since the lifted Abel--Jacobi map $\widetilde{\mathcal A}: \widetilde{\Sigma}_g
		\longrightarrow \mathbb R^{2g}$ is smooth, it is uniformly continuous on every
		compact subset. Therefore, $\widetilde{\mathcal A}\circ\widetilde\phi_n
		\longrightarrow \widetilde{\mathcal A}\circ\widetilde\phi$ uniformly on
		$\overline D$. By definition,
		\[
		\widetilde{\mathbf c}(\phi_n)
		=
		\frac1{\operatorname{Vol}(\Sigma_g)}
		\int_D
		\left(
		\widetilde{\mathcal A}(\widetilde\phi_n(\tilde x))
		-
		\widetilde{\mathcal A}(\tilde x)
		\right)
		\,d\operatorname{Vol}.
		\]
		Hence, $\widetilde{\mathbf c}(\phi_n) \longrightarrow \widetilde{\mathbf
			c}(\phi),$ since uniform convergence of the integrands implies convergence of
		their averages. Now, 
		\[
		\widetilde{\mathcal F}(\phi_n) = \widetilde{\mathcal A}\circ\widetilde\phi_n -
		\widetilde{\mathcal A} - \widetilde{\mathbf c}(\phi_n).
		\]
		The first term converges uniformly on $\overline D$, the second is fixed, and
		the third converges in $\mathbb R^{2g}$. Therefore, $\widetilde{\mathcal
			F}(\phi_n) \longrightarrow \widetilde{\mathcal F}(\phi)$ uniformly on $\overline
		D$. Since each $\widetilde{\mathcal F}(\phi_n)$ is $\Gamma$-invariant, it
		descends to the fluctuation $\mathcal F(\phi_n) \in C^0(\Sigma_g;\mathbb
		R^{2g}),$ and the above convergence is equivalent to uniform convergence on
		$\Sigma_g$: $\mathcal F(\phi_n) \longrightarrow \mathcal F(\phi).$ Finally,
		\[
		\left| \|\phi_n\|_{\mathrm{harm}} - \|\phi\|_{\mathrm{harm}} \right|
		=
		\left| \|\mathcal F(\phi_n)\|_{C^0} - \|\mathcal F(\phi)\|_{C^0} \right| \le
		\|\mathcal F(\phi_n)-\mathcal F(\phi)\|_{C^0},
		\]
		by the reverse triangle inequality for the supremum norm. Since the right-hand
		side tends to zero, we conclude that $\|\phi_n\|_{\mathrm{harm}} \longrightarrow
		\|\phi\|_{\mathrm{harm}}.$ Thus the harmonic norm is continuous with respect to
		the $C^0$ topology.
	\end{proof}
	
	\begin{remark}
		\label{rem:universal_cover_domain}
		In the proof of Theorem~\ref{thm:C0continuity} we use the universal cover
		$\widetilde{\Sigma}_g$ of the closed surface $\Sigma_g$ and a fundamental domain
		$D\subset\widetilde{\Sigma}_g$ for the action of the deck transformation group
		$\Gamma\cong\pi_1(\Sigma_g)$. The fundamental domain is chosen to have compact
		closure and to be connected; its $\Gamma$-translates tile $\widetilde{\Sigma}_g$,
		and the projection $\pi:\widetilde{\Sigma}_g\to\Sigma_g$ restricts to a
		bijection between the interior of $D$ and $\Sigma_g$ up to a set of measure
		zero. Compactness of $\overline{D}$ guarantees that integrals over $D$ are
		finite and that locally uniform convergence on $\widetilde{\Sigma}_g$ implies
		uniform convergence on $\overline{D}$, which is used to pass limits through the
		averaging formula for the constant term $\widetilde{\mathbf c}(\phi)$. The
		base-point $\tilde x_0$ is chosen inside $D$ so that Lemma~\ref{lem:stability_lifts}
		controls the lifts on the whole domain.
	\end{remark}
	
	Here is a consequence of Proposition \ref{prop:lift_characterization} and Theorem
	\ref{thm:C0continuity}. 
	
	\begin{theorem}[Closedness of the Abel--Jacobi kernel]
		\label{thm:closed_AJ_kernel}
		Let $(\Sigma_g,\omega)$ be a closed oriented surface of genus $g\ge2$. Then
		$\Ham(\Sigma_g,\omega)$ is a $C^0$-closed subgroup of $G_\omega(\Sigma_g)$.
	\end{theorem}
	
	\begin{proof}
		Proposition \ref{prop:lift_characterization} states that a symplectic
		diffeomorphism $\phi$ is Hamiltonian if and only if $\widetilde{\mathbf c}(\phi)
		= 0$. In the proof of Theorem~\ref{thm:C0continuity} we showed that the map
		$\phi\mapsto \mathbf c(\phi)$ is continuous with respect to the $C^0$-topology
		on $G_\omega(\Sigma_g)$ and the usual topology on $\mathbb R^{2g}$. Indeed, for
		any sequence $\phi_n\xrightarrow{C^0}\phi$, the lifts were chosen so that
		$\widetilde{\mathbf c}(\phi_n) \to\widetilde{\mathbf c}(\phi)$.	Now let $\{\phi_n\}\subset\Ham(\Sigma_g,\omega)$ be a sequence converging
		uniformly to $\phi\in G_\omega(\Sigma_g)$. By the continuity of the constant
		term,
		\[
		\widetilde{\mathbf c}(\phi) = \lim_{n\to\infty}\widetilde{\mathbf c}(\phi_n).
		\]
		Since each $\phi_n$ is Hamiltonian, Proposition \ref{prop:lift_characterization}
		gives $\widetilde{\mathbf c}(\phi_n)=0$ for all $n$. Hence $\widetilde{\mathbf
			c}(\phi)=0$, and applying the same proposition again yields
		$\phi\in\Ham(\Sigma_g,\omega)$. Thus $\Ham(\Sigma_g,\omega)$ is sequentially
		closed, and because $G_\omega(\Sigma_g)$ is metrizable, it is closed.
	\end{proof}
	
	\begin{theorem}[Local Injectivity of the Harmonic Coordinate Map]
		\label{thm:flux_coordinate_map}
		Let $(\Sigma_g,\omega)$ be a closed oriented surface of genus $g\ge2$. There
		exists a $C^0$-neighborhood $\mathcal U$ of the identity in $G_\omega(\Sigma_g)$
		such that the map
		\[
		\mathcal F(\phi) = \mathcal A\circ\phi-\mathcal A-\mathbf c(\phi),
		\]
		where $\mathbf c(\phi) = \widetilde{\mathbf c}(\phi)\bmod\Gamma \in\Jac(\Sigma_g),$
		is well defined, continuous, and injective on $\mathcal U$.
	\end{theorem}
	
	\begin{proof}
		Fix a lift $\widetilde{\mathcal A}:\widetilde{\Sigma}_g\longrightarrow \mathbb
		R^{2g}$ of the Abel--Jacobi map. For $\phi\in G_\omega(\Sigma_g)$, let
		$\widetilde\phi$ be the normalized lift used in the definition of $\widetilde{\mathbf
			c}(\phi)$. Define the lifted harmonic displacement by
		\[
		\widetilde{\mathcal F}(\phi)(\tilde x) = \widetilde{\mathcal
			A}(\widetilde\phi(\tilde x)) - \widetilde{\mathcal A}(\tilde x) -
		\widetilde{\mathbf c}(\phi).
		\]
		By Lemma~\ref{lem:AJ_decomposition}, $\int_D\widetilde{\mathcal
			F}(\phi)\,d\operatorname{Vol}=0.$ Moreover, by Proposition \ref{DECK}
		$\widetilde{\mathcal F}(\phi)$ descends to a well-defined map $\mathcal
		F(\phi):\Sigma_g\longrightarrow\mathbb R^{2g}.$ The continuity of $\mathcal F$
		follows from the $C^0$-continuity of $\widetilde{\mathbf c}$ and the stability
		of normalized lifts: if $\phi_n\longrightarrow\phi$ in the $C^0$ topology, then
		$\widetilde\phi_n\longrightarrow\widetilde\phi$ uniformly on a fundamental
		domain. Since $\widetilde{\mathcal A}$ is smooth, $\widetilde{\mathcal
			A}\circ\widetilde\phi_n \longrightarrow \widetilde{\mathcal A}\circ\widetilde\phi$
		uniformly. Hence, $\widetilde{\mathcal F}(\phi_n) \longrightarrow
		\widetilde{\mathcal F}(\phi),$ and therefore $\mathcal F(\phi_n)\longrightarrow
		\mathcal F(\phi)$ in $C^0(\Sigma_g;\mathbb R^{2g})$.	We now prove injectivity. Suppose that $\mathcal F(\phi)=\mathcal F(\psi)$ for
		$\phi,\psi\in\mathcal U$. Then, on the universal cover,
		\[
		\widetilde{\mathcal A}(\widetilde\phi(\tilde x)) -
		\widetilde{\mathcal A}(\widetilde\psi(\tilde x)) = \widetilde{\mathbf c}(\phi)
		- \widetilde{\mathbf c}(\psi).
		\]
		Passing to the quotient by the period lattice gives
		\[
		\mathcal A(\phi(x)) = \mathcal A(\psi(x)) + q(\widetilde{\mathbf c}(\phi) -
		\widetilde{\mathbf c}(\psi)),
		\]
		where $q:\mathbb R^{2g}\rightarrow\mathbb R^{2g}/\Gamma$ is the quotient map.
		Since $\phi$ and $\psi$ are diffeomorphisms, $\phi(\Sigma_g)=\psi(\Sigma_g)=\Sigma_g.$
		Therefore,
		\[
		\mathcal A(\Sigma_g) = \mathcal A(\Sigma_g) + q(\widetilde{\mathbf c}(\phi) -
		\widetilde{\mathbf c}(\psi)).
		\]
		By the rigidity of the Abel--Jacobi image, $q(\widetilde{\mathbf c}(\phi) -
		\widetilde{\mathbf c}(\psi)) =0.$ Hence, $\widetilde{\mathbf c}(\phi) -
		\widetilde{\mathbf c}(\psi)\in\Gamma.$ Let $\lambda_\Gamma =
		\inf_{\gamma\in\Gamma\setminus\{0\}} \|\gamma\|>0.$ Because $\widetilde{\mathbf
			c}(\mathrm{id})=0$ and $\widetilde{\mathbf c}$ is continuous, we may shrink
		$\mathcal U$ so that $\widetilde{\mathbf c}(\mathcal U) \subset
		B(0,\lambda_\Gamma/4).$ Then for $\phi,\psi\in\mathcal U$,
		\[
		\|\widetilde{\mathbf c}(\phi) - \widetilde{\mathbf c}(\psi)\| <
		\lambda_\Gamma/2.
		\]
		The only lattice element with norm smaller than $\lambda_\Gamma$ is zero.
		Consequently, $\widetilde{\mathbf c}(\phi) = \widetilde{\mathbf c}(\psi).$ The
		equality of harmonic coordinates now gives $\mathcal A\circ\phi = \mathcal
		A\circ\psi.$ Since $\mathcal A$ is an embedding for $g\ge2$,
		$\phi(x)=\psi(x)$ for every $x\in\Sigma_g$. Therefore, $\phi=\psi.$ Thus
		$\mathcal F$ is injective on $\mathcal U$.
	\end{proof}
	
	\subsection{Fixed Points and the Harmonic Primitive}\label{subsec:fixed_points}
	
	The harmonic decomposition established in Lemma~\ref{lem:AJ_decomposition}
	provides a direct correspondence between the dynamics of a symplectomorphism and
	the geometry of its Abel--Jacobi displacement. In this section we show that, for
	surfaces of genus $g\geq 2$, the fixed point problem can be rewritten as a
	nonlinear level-set problem for the harmonic primitive $\mathcal F(\phi)$.
	
	The resulting formulation separates the displacement of a symplectomorphism into
	two geometrically distinct components:
	\begin{itemize}
		\item the global translation component $\mathbf c(\phi)$, determined by the flux;
		\item the oscillatory component $\mathcal F(\phi)$, measuring the nonlinear
		fluctuation of the displacement.
	\end{itemize}
	A fixed point exists precisely when the oscillatory part compensates for the
	global translation.
	
	\subsubsection{The Harmonic Fixed Point Equation}
	
	Let $\mathcal A:\Sigma_g\longrightarrow \Jac(\Sigma_g)$ be the Abel--Jacobi
	embedding associated with an $L^2$-orthonormal basis of harmonic one-forms.
	Since $g\geq 2$, the Abel--Jacobi map is an embedding. For every $\phi\in
	G_\omega(\Sigma_g)$, Lemma~\ref{lem:AJ_decomposition} gives the decomposition
	\[
	\mathcal A(\phi(x))-\mathcal A(x) = \mathbf c(\phi)+\mathcal F(\phi)(x),
	\]
	in the Jacobian torus. Because $\mathcal A$ is injective, we obtain the
	following equivalence.
	
	\begin{proposition}[Fixed point equation in harmonic coordinates]
		\label{prop:harmonic_fixed_point}
		Let $\Sigma_g$ be a closed surface of genus $g\geq2$. For every symplectomorphism
		$\phi\in G_\omega(\Sigma_g)$, $x\in\operatorname{Fix}(\phi)$ if and only if
		$\mathcal F(\phi)(x)+\mathbf c(\phi)=0$ in $\Jac(\Sigma_g)$. Equivalently,
		\[
		\operatorname{Fix}(\phi) = \left(\mathcal F(\phi)+\mathbf
		c(\phi)\right)^{-1}(0).
		\]
	\end{proposition}
	
	Thus, the fixed point problem becomes a global level-set problem for the
	harmonic primitive.
	
	Now, suppose that $\phi\in\Ham(\Sigma_g,\omega).$ The flux class vanishes, hence
	$\mathbf c(\phi)=0.$ The fixed point equation reduces to $\mathcal
	F(\phi)(x)=0.$ Therefore the fixed points of a Hamiltonian diffeomorphism
	correspond to the zero set of the nonlinear map
	\[
	\mathcal F(\phi): \Sigma_g\longrightarrow \mathbb R^{2g}.
	\]
	
	For a non-degenerate Hamiltonian diffeomorphism, the Arnold conjecture on closed
	surfaces states that $\#\operatorname{Fix}(\phi)\geq 2g+2.$ In harmonic
	coordinates this can be interpreted as the statement that the zero level of the
	harmonic primitive must contain at least $2g+2$ points counted with multiplicity:
	\[
	\#\{x\in\Sigma_g:\, \mathcal F(\phi)(x)=0\} \geq 2g+2.
	\]
	
	\begin{remark}
		The harmonic chart does not replace Floer-theoretic methods in the proof of
		Arnold's theorem. Rather, it gives a geometric reformulation: the topology of
		the surface imposes restrictions on the zero set of the harmonic displacement
		field. This is a reinterpretation, not a new proof, of the Arnold conjecture.
	\end{remark}
	
	\begin{remark}
		The components of $\mathcal F(\phi) = (\mathcal F_1,\ldots,\mathcal F_{2g})$
		are not independent functions. They arise from the symplectic deformation
		equation defining the harmonic chart and therefore satisfy strong compatibility
		relations. The Arnold lower bound may thus be viewed as a topological constraint
		on the zero set of this nonlinear harmonic coordinate map.
	\end{remark}
	
	\subsubsection{Quantitative Obstruction to Fixed Points}
	
	To formulate a quantitative obstruction, we work in the universal cover of the
	Jacobian: $\widetilde{\Jac}(\Sigma_g)=\mathbb R^{2g}.$ Let $\widetilde{\mathbf
		c}(\phi)\in\mathbb R^{2g}$ be a lift of the flux translation and let
	$\widetilde{\mathcal F}(\phi): \Sigma_g\longrightarrow\mathbb R^{2g}$ be the
	corresponding lifted harmonic primitive. Define the flux displacement size by
	\[
	D_{\mathrm{flux}}(\phi) = \operatorname{dist} (\widetilde{\mathbf c}(\phi),\Gamma),
	\]
	where $\Gamma\subset\mathbb R^{2g}$ is the period lattice. Next, define the
	harmonic oscillation by
	\[
	D_{\mathrm{osc}}(\phi) = \sup_{x\in\Sigma_g} \left| \widetilde{\mathcal
		F}(\phi)(x) \right|.
	\]
	
	\begin{proposition}[Quantitative obstruction to fixed points]
		\label{prop:fixed_point_obstruction}
		Let $\phi\in G_\omega(\Sigma_g)$. If $D_{\mathrm{osc}}(\phi) <
		D_{\mathrm{flux}}(\phi),$ then $\phi$ has no fixed points.
	\end{proposition}
	
	\begin{proof}
		Suppose that $x$ is a fixed point. Then $\widetilde{\mathcal F}(\phi)(x) +
		\widetilde{\mathbf c}(\phi) \in\Gamma.$ Hence for some $\gamma\in\Gamma$,
		$$|\gamma-\widetilde{\mathbf c}(\phi)| = |\widetilde{\mathcal F}(\phi)(x)|.$$
		Therefore, $$D_{\mathrm{flux}}(\phi) \le |\widetilde{\mathcal F}(\phi)(x)| \le
		D_{\mathrm{osc}}(\phi),$$ which contradicts the assumption.
	\end{proof}
	
	\begin{corollary}[Necessary harmonic cancellation condition]
		\label{cor:harmonic_cancellation}
		If $\operatorname{Fix}(\phi)\neq\varnothing,$ then $D_{\mathrm{flux}}(\phi) \le
		D_{\mathrm{osc}}(\phi).$ 
	\end{corollary}
	
	The corollary expresses a quantitative principle: a fixed point can exist only
	if the harmonic fluctuation is sufficiently large to compensate for the
	cohomological translation.	Thus, the harmonic norm measures the minimal oscillation required to cancel the
	global flux displacement. Classical Nielsen theory organizes fixed point classes
	according to the homotopy type of a map, whereas the harmonic decomposition
	provides a symplectic refinement of this topological picture. The vector $\mathbf
	c(\phi)$ records the global cohomological displacement generated by the
	symplectic isotopy, while $\mathcal F(\phi)$ measures the pointwise geometric
	correction required to compensate for this translation. The fixed point equation
	is therefore transformed into a balance between a global topological drift and a
	nonlinear harmonic fluctuation.\\
	
	The resulting obstruction gives a metric refinement of the classical fixed point
	problem in the symplectic category. It is important, however, to distinguish
	this global quantitative phenomenon from the local stability of individual fixed
	points. The latter is governed by the usual nondegeneracy condition and is
	inherently a $C^1$-phenomenon. We now reformulate the classical persistence
	theorem for transverse fixed points in harmonic coordinates, thereby expressing
	local fixed point stability within the harmonic coordinate framework.
	
	\begin{corollary}[Persistence of non-degenerate fixed points in harmonic coordinates]
		\label{cor:transverse_stability}
		Let $\phi_0\in G_\omega(\Sigma_g)$ and suppose that $x_0\in\operatorname{Fix}(\phi_0)$
		is a non-degenerate fixed point, that is, $1\notin \operatorname{Spec}(d\phi_0|_{x_0}).$
		Then there exists a $C^1$-neighborhood $\mathcal U$ of $\phi_0$ such that every
		$\phi\in\mathcal U$ possesses a unique fixed point $x(\phi)$ close to $x_0$.
		Moreover, the map $\phi\mapsto x(\phi)$ is $C^1$.
	\end{corollary}
	
	\begin{proof}
		Choose a local coordinate chart $\kappa:V\subset\Sigma_g\longrightarrow\mathbb
		R^2$ around $x_0$. Define $\Psi(\phi,x) = \kappa(\phi(x))-\kappa(x).$ Then
		$\Psi(\phi_0,x_0)=0.$ The derivative in the $x$-direction is
		\[
		d_x\Psi|_{(\phi_0,x_0)} = d\kappa_{x_0} \circ (d\phi_0|_{x_0}-\mathrm{id}).
		\]
		Because $x_0$ is non-degenerate, $d\phi_0|_{x_0}-\mathrm{id}$ is invertible.
		Hence $d_x\Psi|_{(\phi_0,x_0)}$ is an isomorphism. The implicit function
		theorem therefore gives a neighborhood $\mathcal U$ of $\phi_0$ and a unique
		$C^1$ map $x:\mathcal U\rightarrow V$ such that $\Psi(\phi,x(\phi))=0.$ Thus,
		$\phi(x(\phi))=x(\phi),$ and the fixed point persists uniquely.	Finally, by the harmonic fixed point equation, $\widetilde{\mathcal
			F}(\phi)(x(\phi)) +\widetilde{\mathbf c}(\phi) \in\Gamma,$ so the persistent
		fixed point is exactly a persistent zero of the harmonic displacement field.
	\end{proof}
	
	The harmonic primitive acts as a nonlinear coordinate system in which fixed
	points are described by the cancellation of a global cohomological translation
	by a local harmonic deformation.
	
	\begin{example}[Elliptic fixed points and harmonic detection]
		\label{ex:elliptic_fixed_point}
		Consider a Hamiltonian $H: \Sigma_g \to \mathbb{R}$ with a non-degenerate 
		critical point $p_0$ of Morse index $1$ (a saddle point).	Near $p_0$, we can use Darboux coordinates $(x,y)$ such that
		\[
		H(x,y) = \frac{1}{2}(x^2 - y^2) + O(|(x,y)|^3).
		\]
		The Hamiltonian vector field is
		\[
		X_H = x \frac{\partial}{\partial y} - y \frac{\partial}{\partial x} + \text{h.o.t.}
		\]
		The linearization at $p_0$ has eigenvalues $\pm \lambda$ with $\lambda > 0$, 
		making $p_0$ a hyperbolic fixed point. 
		By Subsection~\ref{subsec:fixed_points}, if $\phi = \exp(X_H)$ has a fixed 
		point at $p_0$, then the harmonic fluctuation $\mathcal{F}(\phi)$ must 
		satisfy
		\[
		\mathcal{F}(\phi)(p_0) = -\mathbf c(\phi) \in \mathbb{R}^{2g}.
		\]
		This gives a \emph{global obstruction}: the value of the harmonic primitive 
		at any fixed point is completely determined by the flux. 
		Suppose $\phi$ has two fixed points $p_1, p_2 \in \Sigma_g$. Then
		\[
		\mathcal{F}(\phi)(p_1) = \mathcal{F}(\phi)(p_2) = -\mathbf c(\phi).
		\]
		Since $\mathcal{F}(\phi): \Sigma_g \to \mathbb{R}^{2g}$ is continuous, 
		this implies that the image $\mathcal{F}(\phi)(\Sigma_g)$ must contain 
		the point $-\mathbf c(\phi)$ at least in its interior or boundary. 
		If $\phi \neq \text{id}$ has a fixed point, then
		\[
		\|\phi\|_{\mathrm{harm}} 
		\geq 
		\|\mathbf c(\phi)\|_{\mathbb{R}^{2g}}.
		\]
		Moreover, if $\phi$ is Hamiltonian ($\mathbf c = 0$) and has a fixed point, 
		then $\mathcal{F}(\phi)(p_0) = 0$. By continuity and the intermediate 
		value theorem applied to each component, the harmonic fluctuation must 
		oscillate around zero, giving
		\[
		\|\phi\|_{\mathrm{harm}} 
		= 
		\sup_{p \in \Sigma_g} \|\mathcal{F}(\phi)(p)\|_{\mathbb{R}^{2g}} 
		\geq 
		\text{diam}(\mathcal{F}(\phi)(\Sigma_g))/2.
		\] 
		The harmonic fluctuation map $\mathcal{F}(\phi): \Sigma_g \to \mathbb{R}^{2g}$ 
		acts as a "deformation" of the Abel-Jacobi embedding. The condition 
		$\mathcal{F}(\phi)(p_0) = 0$ at a fixed point means that $p_0$ remains 
		"anchored" at the flux-translated position of the original Abel-Jacobi image. 
		This observation has implications for the Nielsen realization problem in 
		symplectic geometry. If $\phi$ represents a non-trivial element of the 
		mapping class group and is fixed-point free, then either:
		\begin{itemize}
			\item $\|\phi\|_{\mathrm{harm}} > 0$ (always true by 
			Theorem~\ref{thm:intro_nondeg}), or
			\item The harmonic fluctuation provides a continuous "displacement 
			field" with no zeros.
		\end{itemize}
		
		This gives a \emph{quantitative} refinement of Brouwer's fixed point theorem 
		in the symplectic category.
	\end{example}

	\section*{Perspectives}
	
	The harmonic coordinate map developed in this paper suggests possible links
	between symplectic topology, Hodge theory, and geometric quantization. Near the
	identity, the map $\mathcal F$ associates a symplectomorphism to a collection
	of scalar functions determined by harmonic forms. The decomposition, 
	\[
	\mathcal A\circ\phi-\mathcal A = \mathbf c(\phi)+\mathcal F(\phi),
	\]
	naturally separates two different geometric contributions: the flux translation
	in the Jacobian (topological degrees of freedom), and the harmonic primitive
	(pointwise dynamical fluctuation). This resembles the separation of topological
	and dynamical degrees of freedom that frequently appears in geometric
	quantization. One may ask whether this harmonic decomposition can be
	incorporated into a quantization framework, providing a new perspective on the
	functional spaces naturally associated with the symplectomorphism group.
	
	\subsection*{Extension to Topological Area-Preserving Homeomorphisms}
	
	In recent years, $C^0$-symplectic topology has emerged as a major field of
	study, focusing on the group $\mathrm{Homeo}_\omega(\Sigma_g)$ consisting of
	continuous homeomorphisms $f: \Sigma_g \to \Sigma_g$ that preserve the area
	measure $\omega$. This group is the $C^0$-closure of $G_\omega(\Sigma_g)$. 	Because the harmonic decomposition admits a reformulation in terms of
	Abel--Jacobi displacements and continuous orbit integrals, it is natural to ask
	whether it admits a canonical extension to the $C^0$-closure of the
	symplectomorphism group. Such an extension would provide a new functional
	framework for studying topological symplectic groups.
	
	\begin{conjecture}[Topological Harmonic Extension]
		The harmonic displacement map and the associated harmonic norm admit canonical
		continuous extensions to the identity component of $\mathrm{Homeo}_\omega(\Sigma_g)$
		obtained as the $C^0$-closure of $G_\omega(\Sigma_g)$. The extended norm is
		non-degenerate.
	\end{conjecture}
	
	\begin{remark}
		This continuous extension would provide a new candidate invariant for detecting
		topological Hamiltonian homeomorphisms. In particular, it raises the question
		whether the kernel of the extended harmonic flux map coincides with the group of
		Hamiltonian homeomorphisms in the sense of Oh and M\"uller.
	\end{remark}
	
	\begin{remark}[Possible generalizations of the harmonic norm]
		\label{rem:generalizations}
		The non-degeneracy of the harmonic norm appears to be governed by the geometry
		of the harmonic (or Albanese) map rather than by the symplectic structure alone.
		This suggests a broader research programme aimed at determining the class of
		manifolds on which the harmonic norm is non-degenerate.
		
		\begin{enumerate}
			\item \textbf{Closed surfaces of genus $g\ge2$.}
			This is the motivating example of the present work. The Abel--Jacobi map
			$\mathcal A:\Sigma_g\longrightarrow \operatorname{Jac}(\Sigma_g)$ is a smooth
			embedding, and its image admits no non-trivial translation symmetries.
			Consequently, the harmonic norm is non-degenerate (Theorem
			\ref{thm:non_deg_full}).
			
			\item \textbf{Flat tori.}
			For the torus $T^{2n}$, the harmonic $1$-forms are the constant forms
			$dx_1,\ldots,dx_{2n}.$ Every translation $\tau_a(x)=x+a$ satisfies $\mathcal
			A(\tau_a(x))-\mathcal A(x)=a,$ which is constant. Therefore the harmonic
			fluctuation part vanishes identically: $\mathcal F(\tau_a)=0.$ Hence
			translations belong to the kernel of the harmonic norm, showing that
			non-degeneracy fails on flat tori.
			
			\item \textbf{Products with vanishing first Betti number.}
			Consider a product, $M=\Sigma_g\times N,$ where $H^1(N;\mathbb R)=0$. Then
			$H^1(M;\mathbb R)\cong H^1(\Sigma_g;\mathbb R),$ so all harmonic information is
			inherited from the surface factor. It is therefore natural to investigate
			whether the non-degeneracy of the harmonic norm extends to such product
			manifolds.
			
			\item \textbf{Compact K\"ahler manifolds.}
			Every compact K\"ahler manifold admits an Albanese map $\alpha_M:M\longrightarrow
			\operatorname{Alb}(M),$ which generalizes the Abel--Jacobi map of a surface.
			This suggests that the harmonic norm should admit a natural higher-dimensional
			formulation in terms of the Albanese map. The harmonic primitive would then be
			defined using the pullback of harmonic $1$-forms from the Albanese torus.
			
			\item \textbf{Simply connected symplectic manifolds.}
			If $H^1(M;\mathbb R)=0,$ then there are no non-trivial harmonic $1$-forms.
			Consequently, $\|\phi\|_{\mathrm{harm}}=0$ for every symplectomorphism $\phi$,
			so the harmonic norm is identically zero. Typical examples include $S^{2n}$,
			$\mathbb{CP}^n$, and K3 surfaces. In this case, the harmonic norm is maximally
			degenerate and contains no information.
			
			\item \textbf{Nilmanifolds and solvmanifolds.}
			These manifolds generally possess non-trivial harmonic $1$-forms but their
			Albanese maps are rarely embeddings. Understanding the kernel of the harmonic
			norm in this setting may provide further insight into the relationship between
			harmonic geometry and symplectic topology. For instance, on a nilmanifold
			$N/\Gamma$, the harmonic forms are left-invariant, and the Albanese map is
			typically a fibration rather than an embedding. This suggests that the kernel of
			the harmonic norm may be related to the kernel of the Albanese map.
		\end{enumerate}
		
		These examples suggest that the non-degeneracy of the harmonic norm is
		controlled by the geometry of the harmonic (or Albanese) map rather than solely
		by the symplectic structure. This motivates the following conjecture.
		
		\begin{conjecture}[Albanese non-degeneracy]
			Let $(M,\omega)$ be a closed symplectic manifold with $b_1(M)>0$. Assume that
			the Albanese map $\alpha_M:M\rightarrow\operatorname{Alb}(M)$ is an embedding
			and that $\alpha_M(M)$ admits no non-trivial translation symmetries in
			$\operatorname{Alb}(M)$. Then the harmonic norm defined using the pullback of
			harmonic forms from $\operatorname{Alb}(M)$ is non-degenerate on the identity
			component of $\mathrm{Symp}(M,\omega)$.
		\end{conjecture}
	\end{remark}
	
	\section*{Harmonic Persistence and Symplectic Barcodes}
	
	The harmonic decomposition developed in this paper naturally associates to every
	symplectomorphism $\phi\in G_\omega(\Sigma_g)$ the vector-valued harmonic
	displacement
	\[
	\mathcal D_\phi(x) = \mathbf c(\phi)+\mathcal F(\phi)(x),
	\]
	which measures the Abel--Jacobi displacement of the point $x$. The fixed point
	equation $\phi(x)=x$ is therefore equivalent (after fixing a compatible lift of
	the Abel--Jacobi map) to the vanishing of the harmonic displacement:
	\[
	\mathcal D_\phi(x)=0.
	\]
	
	This viewpoint suggests replacing the classical binary question ``Does $\phi$
	possess a fixed point?'' with the quantitative study of the entire harmonic
	displacement landscape.
	
	\subsection*{The harmonic displacement filtration}
	
	Define the non-negative function
	\[
	u_\phi:\Sigma_g\longrightarrow\mathbb R,
	\qquad
	u_\phi(x)=\|\mathcal D_\phi(x)\|.
	\]
	For every $r\ge0$, let
	\[
	X_r(\phi) = \{x\in\Sigma_g:\, u_\phi(x)\le r\}.
	\]
	Since $u_\phi$ is continuous, the family $X_{r_1}(\phi) \subset X_{r_2}(\phi)$
	for $r_1\le r_2$ defines a natural filtration of the surface. When $r=0$, the
	filtration begins at the fixed point set,
	\[
	X_0(\phi) = \operatorname{Fix}(\phi),
	\]
	while for sufficiently large $r$, $X_r(\phi)=\Sigma_g.$	Thus, the filtration records how neighborhoods of the fixed point set grow inside
	the harmonic displacement landscape. Persistent homology associates to every
	filtration a persistence module. Applying persistent homology to $\{X_r(\phi)\}_{r\ge0}$
	therefore produces a family of barcodes.
	
	\begin{definition}[Proposed harmonic barcode]
		The \emph{harmonic barcode} of a symplectomorphism $\phi$ is defined to be the
		persistent homology of the displacement filtration $\{X_r(\phi)\}_{r\ge0}.$
		Equivalently,
		\[
		\mathcal B_{\mathrm{harm}}(\phi) := \operatorname{Barcode} \left(
		\{X_r(\phi)\}_{r\ge0} \right).
		\]
	\end{definition}
	
	At present this should be regarded as a proposed invariant rather than a fully
	developed theory. Nevertheless, it is a natural object associated with the
	harmonic coordinate map and provides a finite-dimensional topological summary of
	the harmonic displacement. Unlike Floer persistence, which is built from the
	action functional on the loop space, the harmonic barcode is constructed
	directly from the surface itself. Consequently, it is expected to be more
	accessible computationally while still encoding information coming from the
	Abel--Jacobi geometry.	One of the fundamental properties of persistent homology is the stability
	theorem, which states that the bottleneck distance between persistence diagrams
	is bounded by the uniform distance between the corresponding filtering
	functions. This immediately suggests the following conjecture.
	
	\begin{conjecture}[Stability of the harmonic barcode]
		Let $\phi,\psi\in G_\omega(\Sigma_g).$ Then
		\[
		d_B \bigl( \mathcal B_{\mathrm{harm}}(\phi), \mathcal B_{\mathrm{harm}}(\psi)
		\bigr) \le \|u_\phi-u_\psi\|_{C^0},
		\]
		where $d_B$ denotes the bottleneck distance between persistence diagrams.
	\end{conjecture}
	
	Since the harmonic coordinate map is continuous in the $C^0$-topology near the
	identity, one expects the harmonic barcode to inherit corresponding stability
	properties.\\
	
	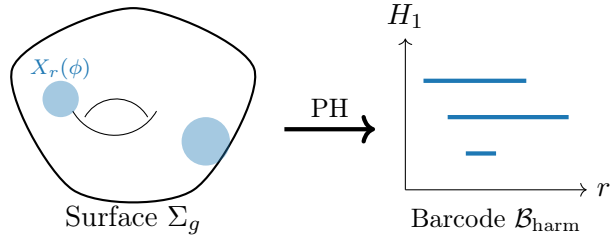
\begin{figure}[htbp]
		\centering
		\begin{tikzpicture}[scale=0.8]
			\draw[thick] plot [smooth cycle] coordinates {(0,0) (2,1) (4,0) (3,-2) (1,-2)};
			\draw (1.2,-0.8) arc (150:30:0.6);
			\draw (1.0,-0.7) arc (-150:-30:0.8);
			\node at (2,-2.5) {Surface $\Sigma_g$};
			
			\fill[color=vertcol, opacity=0.4] (0.8,-0.5) circle (0.3);
			\fill[color=vertcol, opacity=0.4] (3.2,-1.2) circle (0.4);
			\node[color=vertcol, font=\scriptsize] at (0.8, 0) {$X_r(\phi)$};
			
			\draw[->, ultra thick] (4.5, -1) -- (6, -1) node[midway, above, font=\small] {PH};
			
			\draw[->] (6.5, -2) -- (6.5, 0.5) node[above] {$H_1$};
			\draw[->] (6.5, -2) -- (9.5, -2) node[right] {$r$};
			
			\draw[ultra thick, color=vertcol] (6.8, -0.2) -- (8.5, -0.2);
			\draw[ultra thick, color=vertcol] (7.2, -0.8) -- (9.2, -0.8);
			\draw[ultra thick, color=vertcol] (7.5, -1.4) -- (8.0, -1.4);
			
			\node[font=\small] at (8, -2.5) {Barcode $\mathcal{B}_{\text{harm}}$};
		\end{tikzpicture}
		\caption{From Dynamics to Persistence: Sub-level sets of the harmonic displacement $u_\phi(x)$ generate the harmonic barcode.}
	\end{figure}

	Floer theory associates persistence modules to Hamiltonian diffeomorphisms using
	filtrations of the Floer complex by the action functional. The proposed harmonic
	barcode follows a different philosophy. Instead of using the infinite-dimensional
	loop space, it studies the finite-dimensional harmonic displacement function
	\[
	u_\phi(x) = \|\mathbf c(\phi)+\mathcal F(\phi)(x)\|.
	\]
	
	Although considerably simpler, both constructions aim to measure the persistence
	of geometric information under increasing thresholds. Understanding possible
	relationships between harmonic persistence and Floer-theoretic persistence
	appears to be an interesting direction for future research.
	
	The proposed harmonic barcode raises several natural questions:
	
	\begin{enumerate}
		\item Can the harmonic barcode distinguish Hamiltonian from non-Hamiltonian
		symplectomorphisms?
		
		\item Can one compute the barcode explicitly for Dehn twists, pseudo-Anosov
		mapping classes, or linear symplectic maps?
		
		\item Does the harmonic barcode admit a natural extension to $C^0$-symplectic
		topology and area-preserving homeomorphisms?
		
		\item Can analogous constructions be carried out on higher-dimensional
		symplectic manifolds using the Albanese map instead of the Abel--Jacobi map?
		
		\item Is there a conceptual relationship between the harmonic barcode and
		Floer-theoretic barcodes, perhaps through Hodge theory or spectral invariants?
	\end{enumerate}
	
	These questions indicate that harmonic coordinates may provide a natural bridge
	between Hodge theory, symplectic topology, and persistent topology. Whether the
	resulting harmonic barcode yields genuinely new symplectic invariants remains an
	interesting open problem.


\begin{thebibliography}{99}
		
		\bibitem{Tchuaga2021Erratum}
		Tchuiaga, S. (2026). Erratum to ''Hofer-Like Geometry and Flux Theory''.
		Preprint.
		
		\bibitem{Banyaga1997}
		Banyaga, A. (1997).
		\emph{The Structure of Classical Diffeomorphism Groups}.
		Kluwer Academic Publishers.
		
		\bibitem{Tchuaga2018}
		Tchuiaga, S. (2018). On symplectic dynamics.
		\emph{Differ. Geom. Appl.} 61, 170--196.
		
		\bibitem{GriffithsHarris}
		Griffiths, P., \& Harris, J. (1978).
		\emph{Principles of Algebraic Geometry}.
		Wiley.
		
		\bibitem{Earle1975}
		C. J. Earle, ``On the Moduli of Closed Riemann Surfaces with Symplectic
		Automorphisms,'' \emph{Bull. AMS} 81 (1975), pp.~581--586.
		
		\bibitem{MH}
		M. W. Hirsch (1976). \emph{Differential Topology.} Graduate Texts in
		Mathematics 33. Springer, New York.
		
		\bibitem{Petersen2006}
		Petersen, P. (2006).
		\emph{Riemannian Geometry} (2nd ed.).
		Springer-Verlag.
		
		\bibitem{LeRoux2010}
		Le Roux, F. (2010).
		Flux and the $C^0$-topology.
		\emph{Commentarii Mathematici Helvetici}, 85(4), 749--758.
		
		\bibitem{McDuffSalamon}
		McDuff, D., \& Salamon, D. (1998).
		\emph{Introduction to Symplectic Topology} (2nd ed.).
		Oxford University Press.
		
		\bibitem{Polterovich}
		Polterovich, L. (2001).
		\emph{The Geometry of the Group of Symplectic Diffeomorphisms}.
		Birkh\"auser.
		
		\bibitem{HoferZehnder}
		Hofer, H., \& Zehnder, E. (1994).
		\emph{Symplectic Invariants and Hamiltonian Dynamics}.
		Birkh\"auser.
		
	\end{thebibliography}
\end{document}